\documentclass[8pt]{article}
\usepackage{xcolor}
\usepackage{amsmath}
\usepackage{amsthm}
 \usepackage{amssymb}
\usepackage{epsfig}
 \usepackage{leftidx}
\usepackage{graphicx}
\usepackage{hyperref}
\usepackage{amsfonts}
\usepackage{graphicx, color, epstopdf}
 \usepackage{cite}
 \usepackage{bm}
 \usepackage{subfigure}

\newtheorem{theorem}{Theorem}[section]
\newtheorem{lemma}{Lemma}[section]

\newtheorem{definition}{Definition}[section]
\newtheorem{example}{Example}[section]
\newtheorem{conjecture}{Conjecture}[section]
\newtheorem{remark}{Remark}
\renewenvironment{proof}[1][Proof]{\noindent\textit{#1. } }{\hfill$\square$}
\numberwithin{equation}{section}

\newcommand{\Cu}{C_{u,\alpha}}

\def\d{{\;\rm{d}}}

\title{{Do we need decay-preserving error estimate for solving parabolic equations { with  initial} singularity?}
}
\date{}
\author{
Jiwei Zhang\thanks{
School of Mathematics and Statistics, and Hubei Key Laboratory of Computational Science, Wuhan University, Wuhan 430072, China (jiweizhang@whu.edu.cn).}
\and Zhimin Zhang \thanks{
{  Department of Mathematics, Wayne State University, Detroit, MI 48202, USA (ag7761@wayne.edu).} }
\and Chengchao Zhao\thanks{
Beijing Computational Science Research Center, {Beijing 100193,  China} (cheng\_chaozhao@csrc.ac.cn). }
}

\begin{document}
\maketitle

\begin{abstract}

    Solutions exhibiting weak initial singularities arise in various equations, including diffusion and subdiffusion equations. When employing the well-known L1 scheme to solve subdiffusion equations with weak
  singularities, numerical simulations reveal  that this scheme exhibits varying convergence rates for different choices of model parameters (i.e., domain size, final time $T$, and reaction coefficient $\kappa$). This elusive phenomenon is not unique to the L1 scheme but is also observed  in other numerical methods for reaction-diffusion equations such as the backward Euler (IE) scheme, Crank-Nicolson (C-N) scheme, and two-step backward differentiation formula (BDF2) scheme. The existing literature lacks an explanation for the existence of two different convergence regimes, which has puzzled us for a long while and motivated us to study  this inconsistency between the  standard convergence theory and numerical experiences.  In this paper, we provide a general methodology to systematically obtain error estimates that incorporate the exponential decaying feature of the solution. We term this novel error estimate the `decay-preserving error estimate' and apply it to the aforementioned IE, C-N, and BDF2 schemes.  Our decay-preserving error estimate consists of a low-order term with an exponential coefficient  and a high-order term with  an algebraic coefficient, both of which depend on the model parameters.  Our estimates reveal  that  the varying convergence rates are caused by a  trade-off  between these two  components in different model parameter regimes. By considering the model parameters, we capture different states of  the convergence rate that traditional error estimates fail to explain. This approach retains more properties of the continuous solution. We validate our analysis with numerical results.

  \end{abstract}
\section{Introduction}
{  Many} studies have been carried out on  numerical analysis of {the following}  linear reaction subdiffusion equations
\begin{equation}\label{EQ_FPDEs}
\begin{aligned}
\partial_t^\alpha u -\Delta u   &= \kappa u+f, && x\in \Omega,\; t\in(0,T],\\
u(x,t) &=0, && x\in \partial\Omega,\; t\in(0,T],\\
u(x,0) &= u_0(x),&& x\in \bar{\Omega},
\end{aligned}
\end{equation}
where the reaction coefficient~$\kappa$ is a given constant, and $\partial_t^\alpha$ denotes the fractional Caputo's derivative of order $\alpha$ with respect to $t$, namely
\begin{equation*}
\partial_{t}^\alpha u({x},t)  =
\frac{1}{\Gamma(1-\alpha)} \int_0^t\frac{\partial_{s} u({x},s)}{(t-s)^\alpha} ds,\quad  0<\alpha <1. \end{equation*}

Diffusion is one of the most prominent transport mechanisms in nature. The classical diffusion models  generally describe the Brownian motion of particles. {However, over} the last {few} decades, {numerous}  experimental {findings suggest}  that the Brownian motion assumption may not be adequate for accurately describing some physical processes{, such as anomalous diffusion or} non-Gaussian process{. Instead, these processes are} better described by  {the} model \eqref{EQ_FPDEs}. 
{For a comprehensive exploration of various applications and physical modeling, one can refer to the surveys in} \cite{Metzler2014Anomalous,Ralf2000The} and the monograph \cite{JK2011}.

In any numerical methods for solving problem \eqref{EQ_FPDEs}, a key consideration
is that the solution {exhibits}  a {weak singularity} near the initial time $t= 0$ even though the  initial value is smooth. Specifically, if $u_0(x)\in H^1_0(\Omega) \cap H^2(\Omega)$,  the solution {of}  problem \eqref{EQ_FPDEs} with $f = 0$ satisfies \cite{JLZ18}
\begin{equation}
\|\partial_t u(t)\|_{L^2(\Omega)}\leq C t^{\alpha-1}. \label{EQ_Fregularity}
 \end{equation}

Under the regularity condition \eqref{EQ_Fregularity}, the
convergence rate with the maximum norm in time is {proven} to be $\mathcal{O}(\tau^\alpha)$  for various numerical schemes,  such as {the} L1 and L2-type \cite{LLZ18,SOG17,LMZ19,K21} and convolution quadrature schemes \cite{JLZ18, L88}. {Furthermore}, the point-wise error estimate of {the} L1 scheme at time $t=t_n$ is {established as}  $\mathcal{O}(\tau t_n^{\alpha-1})$ in \cite{GOS18,LQZ21}, {implying that the} L1 scheme with {a} smooth initial value is first-order convergent at {the} final time level. {Further discussions can be found} in  \cite{YKF18,JLZ16,JLZ17,MM15}. An example is widely used to numerically verify the theory by taking $\Omega = (0,L)$, and the exact solution of problem \eqref{EQ_FPDEs} is constructed as
\begin{equation} \label{exact}
 u = t^{\alpha}\sin(\pi x/L).
\end{equation}

In the {simulation}, we use the L1 scheme to discretize {the Caputo} derivative with {N} uniformly distributed grid points, and use the finite difference method in space {with $M$} grid points.  Table \ref{T_F1} {displays} temporal convergence rates of {the}  L1 scheme for {a} fixed time level $t_N = T$, $\alpha = 0.5$ and $M=20000$.  {An} interesting and puzzling phenomenon {is observed in} Table \ref{T_F1} that the convergence rates are influenced by {the} choices of model parameters, i.e.,  the interval length $L$, the final time $T$ and  the reaction coefficient~$\kappa$.  Specifically, the convergence order is $\mathcal{O}(\tau)$ while choosing $L=\pi, \kappa = 0,-1$, which is consistent with the existing results of $\mathcal{O}(\tau t_n^{\alpha-1})$ in \cite{GOS18,LQZ21}.  On the other hand, the convergence rate { approaches}  $\mathcal{O}(\tau^{2-\alpha})$ for $T = 10, L=1,\kappa = -8, N = 64$, { consistent} with the results in \cite{LX07,SW06}. 

\begin{table}[!ht]
\begin{center}
\caption {The empirical  convergence rates  of  L1 scheme by taking various parameters $L$, $\kappa$ and $T$.
\label{T_F1} } \vspace*{0.5pt}
\def\temptablewidth{1.0\textwidth}
{\rule{\temptablewidth}{1pt}}
\begin{tabular*}{\temptablewidth}{@{\extracolsep{\fill}}cccccccccc}
&$N$ &\multicolumn{3}{c}{$L=1$}
&\multicolumn{3}{c}{$L=\pi$} \\
\cline{3-5}    \cline{6-8}
&& $\kappa = 1$&$\kappa = 0$ & $\kappa = -8$&$\kappa = 1$&$\kappa = 0$ & $\kappa = -8$\\
\hline
     &64   & 1.28 & 1.29  &1.36 &1.00 &1.06 &1.28 \\
$T=1$ &128  &1.24 &1.25  &1.32 &1.00&1.05&  1.24\\
     &256  &1.19 &1.20  &1.28 &1.00&1.03&  1.19\\
     &512 &1.15&1.16 &1.23 &1.00&1.02&  1.15\\
\hline
     &64  &1.40& 1.41  & 1.45  &1.00&1.16 &1.40 \\
$T=10$ &128  &1.37&1.38 &1.43  &1.00&1.12 &1.37  \\
     &256  &1.33&1.34 &1.40  &1.00&1.09 &1.33  \\
     &512 &1.30&1.31 &1.37 &1.00&1.07 &1.29   \\

\end{tabular*}
{\rule{\temptablewidth}{1pt}}
\end{center}
\end{table}

The results in  Table \ref{T_F1} {indicate} that the different model parameters (i.e., domain sizes, final time $T$, and reaction coefficient $\kappa$) may {lead to varying} convergent rates with respect to $N$.  This  elusive phenomena cannot  be explained  by error estimates in previous literatures{, as their theories do not consider} the influence of model parameters. This inconsistence between numerical experiments and theoretical analysis {has puzzled}  us for a long time, {motivating  further investigation to unveil  the underlying dynamics and solve the mystery}. It is natural to {inquire whether a  similar behavior occurs in} the classical diffusion equation
{  (i.e., $\alpha = 1$ in \eqref{EQ_FPDEs}):}
\begin{equation}\label{EQ_exact}
\begin{aligned}
\partial_t u -\Delta u   &= \kappa u+f, && x\in \Omega,\; t\in(0,T],\\
u(x,t) &=0, && x\in \partial\Omega,\; t\in[0,T],\\
u(x,0) &= u_0(x),&& x\in \bar{\Omega}.
\end{aligned}
\end{equation}

In the {simulation, we use the same exact solution \eqref{exact}  with $\alpha = 0.5$ for \eqref{EQ_exact}}, which also { satisfies} the singularity condition \eqref{EQ_Fregularity}.  {Taking} the implicit Euler scheme as  an example,  Table \ref{T_P1} shows that the convergence rate is $\mathcal{O}(\tau^{\alpha})$ for $L=\pi, T = 1${,} and various $\kappa$ {values, while it is } $\mathcal{O}(\tau)$ for $L=1$ and different $T$ and $\kappa$. On the other hand  for $L=\pi$ and $T=10$, the convergence rate is $\mathcal{O}(\tau)$ for $\kappa = 0,-8$, and is changing to $\mathcal{O}(\tau^\alpha)$ for $\kappa = 1$. {Similar  phenomena are observed in simulations using} C-N and BDF2 schemes{, where } the convergence rates change with different model parameters. The results in Tables \ref{T_F1} and \ref{T_P1} {suggest}  that it is a common phenomenon {for}  both diffusion and subdiffusion equations when an initial singularity is {present}.

%

\begin{table}[!ht]
\begin{center}
\caption {The empirical   convergence rates of implicit Euler scheme    by taking various $L$, $\kappa$ and $T$.
\label{T_P1} } \vspace*{0.5pt}
\def\temptablewidth{1.0\textwidth}
{\rule{\temptablewidth}{1pt}}
\begin{tabular*}{\temptablewidth}{@{\extracolsep{\fill}}cccccccccc}
&$N$ &\multicolumn{3}{c}{$L=1$}
&\multicolumn{3}{c}{$L=\pi$} \\
\cline{3-5}    \cline{6-8}
    && $\kappa = 1$&$\kappa = 0$ & $\kappa = -1$&$\kappa = 1$&$\kappa = 0$ & $\kappa = -1$\\
\hline
            &64   &1.03   & 1.03  &1.02&0.47 &0.50&0.56 \\
$T=1$ &128  &1.00 &1.01  &1.01 &0.48&0.50&  0.54\\
           &256  &0.98 &0.99  &1.00&0.49&0.50&  0.53\\
            &512 &0.96&0.98 &1.00 &0.49&0.50&  0.52\\
\hline
              &64  &1.01& 1.01 & 1.01 &0.47&1.03 &1.01 \\
$T=10$ &128  &1.01&1.01 &1.01  &0.48&1.01 &1.01  \\
             &256  &1.00&1.00 &1.00  &0.49&1.00 &1.00  \\
             &512 &1.00&1.00&1.00 &0.49&0.99 &1.00   \\

\end{tabular*}
{\rule{\temptablewidth}{1pt}}
\end{center}
\end{table}

The existing theory {fails to explain the observed variations in} convergence rates with different model parameters. This { discrepancy} motivates us to develop a methodology for systematically describing various convergence regimes{, spanning from lower to high orders}. To the end, we first eliminate the effect of the spatial domain {by considering} the corresponding ordinary differential equations (ODEs) with an initial singularity. {We establish refined} decay-preserving  error estimates {for} the widely used implicit Euler scheme, C-N scheme, and BDF2 scheme. Specifically for $\kappa <0$, the  novel decay-preserving error estimates are proven to be
\begin{align}
|e^n|\leq e^{C\kappa t_n}|e^{0}|+ \Cu \Big( e^{C\kappa t_n}\tau^\alpha+Ct_{n-1}^{\alpha-1}\tau^k\Big),\label{EQ_INode11}
\end{align}
where $k=1$ represents for implicit Euler scheme and $k=2$ for  C-N and BDF2 schemes. {Afterwards, leveraging} the corresponding  eigenvalue problem, we extend the refined decay-preserving  error estimates {to}  diffusion problems with {an} initial singularity. {Subsequently, we conduct} numerical experiments to further demonstrate and {validate} our decay-preserving error estimates{, which can effectively} capture  the { characteristics} of various convergence regimes for different  $L, \kappa$, and $T$. { Consequently, we establish} an analyzable {connection} between the model parameters and the {primary} features of convergence rates{. Furthermore, we put forward a conjecture regarding} the preserving error estimate for subdiffusion equations.

{The organization of this paper is as follows. In Section \ref{Sec_2},  we {present}  a general methodology for systematically obtaining new decay-preserving error estimates {for widely} used IE, C-N, and BDF2 schemes {applied}  to ODEs \eqref{EQ_ordinary}. 
{In Section  \ref{Sec_3}, by incorporating} the eigenvalue problem \eqref{EQ_eigen}, {we extend these} decay-preserving error estimates  to reaction-diffusion equations \eqref{EQ_exact}. {Section  \ref{Sec_4} proposes a  conjecture for} the decay-preserving error estimate for reaction-subdiffusion equation \eqref{EQ_FPDEs}. Based on these decay-preserving error estimates,  we further  discuss the {connections} between convergence regimes and model parameters{. Furthermore, we} provide numerical examples to { illustrate and validate} our theoretical results. {Section 5 contains} detailed proofs of {the} main results, and the paper {concludes with a summary} in Section 6.}

\section{Decay-preserving  error estimate for  { ODE solvers}\label{Sec_2}}
We first present the error  estimates {for} widely used schemes, {namely}, the implicit Euler scheme, Crank-Nicolson (C-N) scheme and BDF2 scheme, {applied to solve}  the following ODEs
 \begin{equation}\label{EQ_ordinary}
\begin{aligned}
&u_t = \kappa u +f, \quad t>0
\end{aligned}
\end{equation}
with the initial value $u(0)= u_0$ and a given constant $\kappa\in \mathbb{R}$. Assume {that} the solution of \eqref{EQ_ordinary} satisfies the following regularity
\begin{equation}\label{Assump_ode}
      	 |\partial^k u| \leq C_{u,\alpha} t^{ -k + \alpha },\ k = 1,2,3  \text{ and }  \;  0<\alpha <1.
       \end{equation}
As the first step, we generate a grid by
\begin{equation} \label{grid}
t_n = n\tau, \quad n = 0,1,\dots,N, \quad  \tau = T/N.
\end{equation}
At a grid point, we take $U^n$ as the approximate value of the true solution of $u^n=u(t_n)$.  For any time sequence $\{v^n\}_{n=0}^N$, we introduce the following notations
 \begin{align}
\nabla_\tau v^n  = v^n-v^{n-1}, \quad
v^{n-\frac12} = \frac12(v^n+v^{n-1}), \quad D_2 v^n := \frac{3v^n-4v^{n-1}+v^{n-2}}{2\tau}.
\end{align}
The discretizations are given by
 \begin{align}
&\frac{1}{\tau}\nabla_\tau U^n = \kappa U^n+f^n,   &&\text{Implicit Euler scheme}\label{EQ_IEode} \\
&\frac{1}{\tau} \nabla_\tau U^n= \kappa U^{n-\frac12}+f^{n-\frac12},  && \text{C-N scheme} \label{EQ_CNode}\\
& D_2 U^n=\kappa U^{n}+f^{n}, \;\;n\geq 2,  && \text{BDF2 scheme}
\label{EQ_BDF2ode}
\end{align}
with initial value $U^0 = u_0$. Noting BDF2 needs two initial values, we here use IE scheme to compute $U^1$.

Before the numerical analysis of schemes \eqref{EQ_IEode},\eqref{EQ_CNode} and \eqref{EQ_BDF2ode}, we first present a lemma which will be {used frequently in later} proofs.
\begin{lemma}\label{Lemma_SCN}
Given constants $\varrho>1$ and $\beta<-1$, let
 \begin{align}
I_n & =  \sum_{k=2}^{n-1}\tau \varrho^{-(n-k-1)}t_{k}^\beta , \label{EQ_1347}
\end{align}
where $t_k$ is defined in \eqref{grid}.
 Then it holds for $n\geq 4$ that
\begin{align}
I_n&\leq \frac{-1}{\beta+1}\Big((2^{-(\beta+1)}(1-\varrho^{-(\frac{n}{2}-1)})-1)t_{n-1}^{\beta+1}+\varrho^{-(\frac{n}{2}-2)}\tau^{\beta+1}
\Big).\label{EQ_0008}
\end{align}
\end{lemma}
\begin{proof}
Applying the inequality ${t_{k-1}^{\beta+1}-t_{k}^{\beta+1}}\geq -(\beta+1){\tau}t_k^\beta$ to \eqref{EQ_1347}, one has
\begin{align}
I_n&\leq \frac{-1}{\beta+1}\Big(\sum_{k=2}^{n-1}\varrho^{-(n-k-1)}(t_{k-1}^{\beta+1}-t_{k}^{\beta+1})\Big)\nonumber\\
&=\frac{-1}{\beta+1}\Big((\varrho-1)\sum_{k=2}^{n-2}\varrho^{-(n-k-1)}t_{k}^{\beta+1}+\varrho^{-(n-3)}\tau^{\beta+1}
-t_{n-1}^{\beta+1}\Big).\label{EQ_2345}
\end{align}
Set the integer $n_0 = [\frac{n}{2}]$. It is easy to check $t_{n_0}\leq \frac{t_n}{2}\leq t_{n_0+1}$ and
\begin{align}
(\varrho-1)\sum_{k=2}^{n-2}\varrho^{-(n-k-1)}t_{k}^{\beta+1}&= (\varrho-1)\big(\sum_{k=2}^{n_0}\varrho^{-(n-k-1)}t_{k}^{\beta+1}+\sum_{k=n_0+1}^{n-2}\varrho^{-(n-k-1)}t_{k}^{\beta+1}\big)\label{EQ_1607}\\
&\leq \tau^{\beta+1}(\varrho^{-(n-n_0-2)}-\varrho^{-(n-3)})+t_{n_0+1}^{\beta+1}(1-\varrho^{-(n-n_0-2)})\nonumber\\
&\leq \tau^{\beta+1}(\varrho^{-(\frac{n}{2}-2)}-\varrho^{-(n-3)})+(\frac{t_{n}}{2})^{\beta+1}(1-\varrho^{-(\frac{n}{2}-1)}).\label{EQ_2346}
\end{align}
The proof is completed by inserting \eqref{EQ_2346} into \eqref{EQ_2345}.
\end{proof}
\begin{remark}
 In summation  \eqref{EQ_1347},   the weight of $t_2$ is $\varrho^{-(n-3)}${,}  and the weight of $t_{n-1}$ is 1{. This implies that} the weights are in different scales for different $t_k$ when $n\gg 1$. The aim of introducing an integer $n_0$ is to  divide the summation into two parts in \eqref{EQ_1607}{, allowing one to} estimate  the two scales in \eqref{EQ_0008}, respectively. In other words, the weight of {the} leading order { exponentially decays, while} the second order (constant quantity) {algebraically decays} with respect to $n$.  The two-scale estimate in Lemma \ref{Lemma_SCN} will play a key role in { proving} our {decay-preserving} error estimates.
\end{remark}

\subsection{Decay-preserving  error estimate for implicit Euler scheme}
We consider the case of $\kappa < 0$.  It follows from \eqref{EQ_IEode} that
\begin{equation*}
(1-\kappa\tau)U^n=U^{n-1}+\tau f^n,
\end{equation*}
which yields
\begin{equation*}
U^n=(1-\kappa\tau)^{-n}U^{0}+\tau  \sum_{k=1}^n (1-\kappa\tau)^{-(n+1-k)}f^k.
\end{equation*}
Set the error $e^n = u^n-U^n$. The error satisfies the following equation
\begin{equation}\label{EQ_1709}
e^n=(1-\kappa\tau)^{-n}e^{0}+\tau \sum_{k=1}^n (1-\kappa\tau)^{-(n+1-k)}R^k,
\end{equation}
with the truncation error
$R^k = -\frac1{\tau_k}\int_{t_{k-1}}^{t_k}(s-t_{k-1})u''(s)\d s.$
Following regularity condition \eqref{Assump_ode}, one has
\begin{equation}
|R^1|\leq \tau^{\alpha-1} \;\text{ and } \;|R^k| \leq \Cu \tau t_{k-1}^{\alpha-2},\quad k\geq 2.\label{EQ_1606}
\end{equation}

\begin{lemma}\label{Pro_1}
For any $0<\tau\leq -1/\kappa$ and $\upsilon\geq 0$, it holds that
\begin{equation} \label{inq1}
e^{\kappa \upsilon\tau}\leq (1-\kappa \tau)^{-\upsilon}\leq e^{\frac{\kappa \upsilon\tau}{2}}.
\end{equation}
\end{lemma}
\begin{proof}
For any $x\geq 0$, the direct calculation shows that
\[
   x-\frac{x^2}{2} \leq \ln(1+x) \leq x \quad \text{and} \quad (1+x)^{1/x} = e^{\frac{\ln(1+x)}{x}},
\]
  which implies $  e^{1-x/2}\leq  (1+x)^{1/x} \leq e. $
If $0\leq x\leq 1$, we further have
$e^{1/2}\leq (1+x)^{1/x} \leq e$.

Thus, setting $x = -\kappa\tau\leq 1$ and noting $(1-\kappa \tau)^{-\upsilon} = \big((1-\kappa \tau)^{-\frac{1}{\kappa \tau}}\big)^{\kappa \upsilon\tau}$,
 we immediately have the inequalities in \eqref{inq1}.
The proof is completed.
\end{proof}

\begin{theorem}\label{TH_IEode}
Let $u^n$ and $U^n$ be solutions to \eqref{EQ_ordinary} and \eqref{EQ_IEode} respectively. Then for $ -\kappa\tau <1$  it  holds
\begin{equation}
|e^n|\leq e^{\kappa t_n/2}|e^{0}|+ \Cu \Big(  C_1e^{\kappa t_n/4}\tau^\alpha+C_2t_{n-1}^{\alpha-1}\tau\Big),\label{EQ_1536}
\end{equation}
where $C_1:=\frac2{1-\alpha}+3e^{\frac{\kappa t_n}{4}}$ and
\begin{equation*}
C_2:=\Big\{\begin{array}{cc}
\frac{1}{1-\alpha}(2^{1-\alpha}(1-e^{\frac{\kappa  t_n}{4}})-1), & t_n >\frac4{\kappa}\ln(1-2^{\alpha-1}),\\
0, &t_n \leq\frac4{\kappa}\ln(1-2^{\alpha-1}).
\end{array}
\end{equation*}
\end{theorem}
\begin{proof}
Inserting \eqref{EQ_1606}   into  \eqref{EQ_1709}, one has
\begin{equation}
|e^n|\leq (1-\kappa\tau)^{-n}|e^{0}|+ \Cu \Big(S_n\tau + (1-\kappa\tau)^{-n}(2-\kappa\tau)\tau^\alpha\Big),\label{EQ_1903}
\end{equation}
where $S_n$ is defined by
\begin{equation*}
S_n \;=\; (1-\kappa\tau)^{-(n+1)}\sum_{k=3}^{n}\tau (1-\kappa\tau)^{k}t_{k-1}^{\alpha-2}
 \;=\; (1-\kappa\tau)^{-1}\sum_{k=2}^{n-1}\tau(1-\kappa\tau)^{-(n-k-1)}t_k^{\alpha-2}.
\end{equation*}
Taking $\varrho=1-\kappa\tau, \beta  = \alpha-2  $, it follows from  Lemma \ref{Lemma_SCN} that
\begin{align*}
S_n&\leq \frac{1}{(1-\alpha)(1-\kappa\tau)}\Big((2^{1-\alpha}(1-(1-\kappa\tau)^{-(\frac{n}{2}-1)})-1)t_n^{\alpha-1}+(1-\kappa\tau)^{-(\frac{n}{2}-2)}\tau^{\alpha-1}
\Big).
\end{align*}
Substituting the above bound of $S_n$ into \eqref{EQ_1903}, one has
\[
|e^n|\leq (1-\kappa\tau)^{-n}|e^{0}|+ \Cu \Big(  C_1(\tau,n)(1-\kappa\tau)^{-\frac{n}{2}}\tau^\alpha+C_2(\tau,n) t_{n-1}^{\alpha-1}\tau\Big),
\]
where
$C_1(\tau,n):=\frac{1-\kappa\tau}{1-\alpha}+ (2-\kappa\tau)(1-\kappa\tau)^{-\frac{n}{2}},\;
C_2(\tau,n):=\frac{1}{1-\alpha}(2^{1-\alpha}(1-(1-\kappa\tau)^{-\frac{n}{2}})-1). $

Noting $\tau\leq -1/\kappa$ and using Lemma \ref{Pro_1}, we have
\begin{align*}
C_1(\tau,n)\leq\frac2{1-\alpha}+3e^{\frac{\kappa  t_n}{4}}:=C_1,\qquad
C_2(\tau,n)\leq \frac{1}{1-\alpha}(2^{1-\alpha}(1-e^{\frac{\kappa  t_n}{4}})-1)\leq C_2.
\end{align*}
Thus, the proof is completed.
\end{proof}



\subsection{Decay-preserving error estimate for C-N scheme}
We now consider the decaying error estimate for C-N scheme with $\kappa<0$.  It follows from \eqref{EQ_CNode} that
\begin{equation*}
(1-\kappa\tau/2)U^n=(1+\kappa\tau/2)U^{n-1}+\tau f^{n-\frac12},
\end{equation*}
which yields
\begin{equation}
U^n=\psi_\tau^{-n}U^{0}+\tau  (1-\frac{\kappa\tau}{2})^{-1} \sum_{k=1}^n \psi_\tau^{-(n-k)}f^{k-\frac12},
\end{equation}
where $\psi_\tau :=(1-\frac{\kappa\tau}{2})/(1+\frac{\kappa \tau}{2})$.
Let $e^n = u^n-U^n$. The error satisfies
\begin{align}
e^n=\psi_\tau^{-n}e^{0}+\tau  (1-\frac{\kappa\tau}{2})^{-1}\sum_{k=1}^n \psi_\tau^{-(n-k)}(R_1^{k}+R_2^k),\label{EQ_1421}\end{align}
where the truncation errors are respectively given by
$R_1^k := \nabla_\tau U^k/\tau-u'(t_{k-\frac12}) \; \text{and} \; R_2^k := U^{k-
\frac12}-u(t_{k-\frac12}).$
The direct calculation shows that the truncations errors can be exactly expressed by
\begin{align*}
R_1^k &=\frac1{2\tau}\big(\int_{t_{k-\frac12}}^{t_k}(t_k-s)^2u'''(s)\d s+\int^{t_{k-\frac12}}_{t_{k-1}}(s-t_{k-1})^2 u'''(s) \d s\big), \\
R_2^k &=\frac1{2}\big(\int_{t_{k-\frac12}}^{t_k}(t_k-s)u''(s)\d s+\int^{t_{k-\frac12}}_{t_{k-1}}(s-t_{k-1}) u''(s) \d s\big).
\end{align*}
Using the regularity condition \eqref{Assump_ode}, one has the following error bounds
\begin{equation}
R_1^k+R_2^k\leq \Big\{\begin{array}{cc}C_{u,\alpha}\tau^{\alpha-1},& k = 1,\\
C_{u,\alpha} \tau^2 t_{k-1}^{\alpha-3},& k\geq 2.\label{EQ_1420}
\end{array}
\end{equation}
\begin{lemma}\label{Pro_2}
For any $0<-\kappa\tau\leq 1$ and $\upsilon\geq0$, it holds that
\begin{equation}e^{\frac76\kappa\upsilon \tau}\leq  \psi_\tau^{-\upsilon}\leq e^{\kappa\upsilon \tau},
\end{equation}
where $\psi_\tau =(1-\frac{\kappa\tau}{2})/(1+\frac{\kappa \tau}{2})$.
\end{lemma}
\begin{proof}
Noting that $
    (1+x)^{\frac1x} = e^{\frac{\ln(1+x)}{x}}=e^{1-\frac1{x}\int_0^x \frac{(x-s)}{(1+s)^2}\d s}\; \text{and} \;(1-x)^{\frac1x} = e^{-1-\frac1{x}\int_0^x \frac{(x-s)}{(1-s)^2}\d s}, \forall 0<x< 1,$
 we then have
$ \big(\frac{1-x}{1+x}\big)^{\frac1{x}} = e^{-2-\frac1{x}\int_0^x \frac{4s(x-s)}{(1-s^2)^2}\d s}.$
For any $0<x\leq \frac12$, note that \begin{align*}
\frac1{x}\int_0^x \frac{4s(x-s)}{(1-s^2)^2}\d s &\leq \frac{4}{(1-x^2)^2x}\int_0^x s(x-s)\d s  = \frac2{3(\frac1{x}-x)^2}<\frac13,
\end{align*}
 we further have
    \begin{align}\label{id2}
  e^{-\frac73}\leq \Big(\frac{1-x}{1+x}\Big)^{\frac1{x}}\leq e^{-2}.
   \end{align}
By the definition of
  \begin{equation}
  \psi_\tau^{-\upsilon} = \Big(\big(\frac{1+\frac{\kappa \tau}{2}}{1-\frac{\kappa\tau}{2}}\big)^{-\frac2{\kappa \tau}}\Big)^{-\frac{\kappa \upsilon \tau}{2}},
  \end{equation}
 we have $e^{\frac76\kappa\upsilon \tau}\leq \psi_\tau^{-\upsilon}\leq e^{\kappa\upsilon \tau}$ in analogy of \eqref{id2}.
The proof is completed.
\end{proof}

\begin{theorem}\label{TH_CNode}
Let $u^n$ and $U^n$ be solutions to \eqref{EQ_ordinary} and \eqref{EQ_CNode} respectively. Then for $ 0<-\kappa\tau <1$  it holds
\begin{align}
|e^n|\leq e^{\kappa t_n}|e^{0}|+ C_{u,\alpha} \Big(C_3e^{\frac{\kappa  t_n}{2}}\tau^\alpha+C_4t_n^{\alpha-2}\tau^2\Big),\label{EQ_0114}
\end{align}
where $C_3:=\frac{9}{2-\alpha}+12e^{\frac{\kappa  t_n}{2}}$ and
\begin{align*}
C_4:=\Big\{\begin{array}{cc} \frac1{2-\alpha}(2^{2-\alpha}(1-e^{\frac{7\kappa t_n}{12}})-1), & t_n> \frac{12}{7\kappa}\ln(1-2^{\alpha-2}),\\
0, & t_n\leq  \frac{12}{7\kappa}\ln(1-2^{\alpha-2}).
\end{array}
\end{align*}
\end{theorem}
\begin{proof} Inserting the truncation error bounds \eqref{EQ_1420} into \eqref{EQ_1421}, one has
\begin{equation}
|e^n| \;\leq \; \psi_\tau^{n}|e^{0}|+ C_{u,\alpha} \Big(S_n\tau^2 + (\psi_\tau+1)\psi_\tau^{-(n-1)}\tau^\alpha\Big),\label{EQ_0030}
\end{equation}
where \begin{equation}
S_n  \;=\; \sum_{k=3}^{n}\tau \psi_\tau^{-(n-k)}t_{k-1}^{\alpha-3} \;=\;  \sum_{k=2}^{n-1}\tau \psi_\tau^{-(n-k-1)}t_{k}^{\alpha-3}.\label{EQ_2331}
\end{equation}
The case of $n\leq 3$ can be easily estimated and here we only assume $n\geq 4$.

For $0<-\kappa\tau<1$, taking $\varrho = \psi_\tau$ and $\beta = \alpha -3$ in Lemma \ref{Lemma_SCN}, and applying Lemma \ref{Pro_2}  to \eqref{EQ_2331}, we have
\begin{equation} \label{id3}
S_n \; \leq \;\frac{1}{2-\alpha}\Big((2^{2-\alpha}(1-e^{\frac{7\kappa t_n}{12}})-1)t_n^{\alpha-2}+9e^{\frac{\kappa  t_n}{2}}\tau^{\alpha-2}\Big),
\end{equation}
where $\psi_\tau \leq 3$ is used. Again, from Lemma \eqref{Pro_2} and the assumption $-\kappa\tau<1$, we have
\begin{equation}
(\psi_\tau+1)\psi_\tau^{-(n-1)}\leq 12e^{\kappa t_n}.\label{EQ_1717}
\end{equation}
The proof is completed by inserting \eqref{id3} and \eqref{EQ_1717} into  \eqref{EQ_0030}.
\end{proof}
\subsection{Decay-preserving  error estimate for BDF2 scheme}
We now consider the decaying error estimate for BDF2 scheme with $\kappa<0$. From \eqref{EQ_BDF2ode}, one has
\begin{align}
U^1-U^{0}&=\kappa \tau U^{1}+\tau f^{1}, \label{BDF1}\\
\frac32U^n-2U^{n-1}+\frac12U^{n-2}&=\kappa\tau U^{n}+\tau f^{n}, \quad n\geq 2.
\label{BDF2}
\end{align}
Set $F^k  = f^k(k\geq 3)$ and $F^2 = f^2-\frac{U^0}{2\tau} ,\;F^1 = f^1+\frac{U^0}{\tau}$. We further introduce the following notations, say the BDF2 kernels, as $A^{(1)}_0 =1-\kappa\tau$ and for $n\geq 2$
\begin{align}
A^{(n)}_0 =\frac32-\kappa\tau, \quad  A^{(n)}_1 = -2,\quad  A^{(n)}_2 = \frac12,\quad A^{(n)}_j = 0 \;( 3\leq j <n).
\label{EQ_BDF2kernels}
\end{align}
 Thus, BDF2 scheme \eqref{BDF1}-\eqref{BDF2} can be reformulated into the following convolution form
\begin{equation*}
\sum_{j = 1}^k A^{(n)}_{n-k} U^j = \tau F^{k}.
\end{equation*}
Set the error $e^n = u^n-U^n$. The error satisfies the {governing} equation as
\begin{equation}
\sum_{j = 1}^k A^{(k)}_{k-j} e^j = \tau R^{k},\label{EQ_CF}
\end{equation}
where
\begin{align}
R^1 = \eta^1+e^0/\tau,\; R^2 = \eta^2- e^0/(2\tau) \quad \text{and} \quad R^k  = \eta^k\; (k\geq 3).\label{EQ_RRR}
\end{align}
 Under the regularity condition \eqref{Assump_ode},  as discussed in \cite{ZZ21}, the truncation error $\eta^k:=D_2 u^k-u_t(t_k)$ satisfies
 \begin{equation}
 |\eta^k| \leq C_{u,\alpha}\tau^{\alpha-1}  \; ( k=1,2)  \quad \text{ and } \quad |\eta^k| \leq C_{u,\alpha} t_{k-2}^{\alpha-3}\tau^2,\quad\; k\geq 3.\label{EQ_0018}
 \end{equation}
The decay-preserving error estimate of BDF2 scheme is given as follows.
\begin{theorem}\label{TH_BDF2ode}
Let $u^n$ and $U^n$ be solutions to \eqref{EQ_ordinary} and \eqref{EQ_BDF2ode}, respectively. It holds  for $ 0<-4\kappa\tau <1$ that
\begin{align}
|e^n|\leq e^{\frac{\kappa  t_n}{2}}|e^{0}|+ 2C_{u,\alpha} \Big(C_1e^{\frac{\kappa  t_n}{2}}\tau^\alpha+C_2t_{n-1}^{\alpha-2}\tau^2\Big),\label{EQ_0026}
\end{align}
where $C_5:=\frac{2}{2-\alpha}+4e^{\frac{\kappa  t_n}{4}}$ and
\begin{align*}
C_6:=\Big\{\begin{array}{cc} \frac1{2-\alpha}(2^{2-\alpha}(1-e^{\frac{\kappa  t_n}{2}})-1)& t_n> \frac{2}{\kappa}\ln(1-2^{\alpha-2}),\\
0& t_n\leq  \frac{2}{\kappa}\ln(1-2^{\alpha-2}).
\end{array}
\end{align*}
\end{theorem}
{For brevity, we {defer} the proofs of error estimates for {the} BDF2 scheme to section 5.1, and  also {postpone}  the proofs of the {subsequent}  error estimates for PDEs  to section 5.2. }

\section{{ Decay-preserving  error estimates for  PDEs: Main results}\label{Sec_3}}
We now consider error estimates of  the three schemes for {the} diffusion equation \eqref{EQ_exact}, {assuming that} the solution to \eqref{EQ_exact} satisfies  {the} regularity condition \eqref{EQ_Fregularity}.  To  {proceed}, we introduce the  associate eigenvalue problem
\begin{align}
-\Delta \omega = \lambda \omega \quad \text{with} \quad \omega = 0 \;\text{on}\; \partial \Omega.\label{EQ_eigen}
\end{align}
This eigenvalue problem admits a  sequence of eigenvalues $\{\lambda_k\}_{k=1}^\infty$ and a corresponding sequence of eigenfunctions $\{\omega_k\}_{k=1}^\infty$.  As {demonstrated} in \cite[Chapter 3]{G06}, the eigenvalues $\{\lambda_k\}_{k=1}^\infty$  are nondecreasing, positive and  tend to infinity  as $k\to \infty$. The eigenfunctions $\{\omega_k\}_{k=1}^\infty$ form an orthonormal basis in $L^2(\Omega)$ such that  {any}  $v \in L^2(\Omega)$ {can be represented as follows:} 
\begin{align}
v = \sum_{k=1}^\infty (v,\omega_k)\omega_k.\label{EQ_basis}
\end{align}
Taking $\Omega = (0,L)$ for an example, the eigenvalues and {corresponding eigenvectors are} given as
\begin{equation}
\lambda_k = (k\pi/L)^2, \quad \omega_k(x) = \sin(k\pi x/L), \qquad k = 1,2,\cdots.
\end{equation}

\subsection{Decay-preserving  error estimate of the implicit Euler scheme}
The implicit Euler semi-discrete scheme is given by
\begin{equation}\label{EQ_pdeIE}
\begin{aligned}
\frac{1}{\tau} \nabla_\tau U^n&= \Delta U^n+\kappa U^n+f^n,&&x\in \Omega, \;\;\;1\leq n\leq N\\
U^0(x) &= u_0(x),&&x\in\Omega\\
U^n(x)&=0, &&x\in\partial \Omega, \;1\leq n\leq N.
\end{aligned}
\end{equation}
Denote the error $e^n:=u^n-U^n\; (0\leq n\leq N)$.
The error $e^n$ satisfies
\begin{equation}\label{EQ_pdeIEerror}
\begin{aligned}
\frac{1}{\tau}\nabla_\tau e^n &= \Delta e^n+\kappa e^n +R^n,&&x\in \Omega, \;\;\;1\leq n\leq N\\
e^0(x) &= 0,&&x\in \Omega\\
e^n(x)&=0, &&x\in \partial \Omega, \;1\leq n\leq N.
\end{aligned}
\end{equation}
where  $R^n = -\frac1{\tau_n}\int_{t_{n-1}}^{t_n}(s-t_{n-1})u_{tt}(s)\d s$.

\begin{theorem}\label{TH_IEpde}
Let $u^n$ {and}  $U^n$ be the solutions of \eqref{EQ_exact} and \eqref{EQ_pdeIE}{, }respectively. Let $\lambda_1$ be the minimal eigenvalue of {the} eigenvalue problem \eqref{EQ_eigen}{, and let $\kappa<\lambda_1$. Then,} for $ \tau <1/(\lambda_1-\kappa)$, it holds that
\begin{align}
\|e^n\|\leq e^{-(\lambda_1-\kappa) t_n/2}\|e^{0}\|+ C_{u,\alpha} \Big(  C_7e^{-(\lambda_1-\kappa) t_n/4}\tau^\alpha+C_8t_{n-1}^{\alpha-1}\tau\Big),
\end{align}
where $C_7:=\frac2{1-\alpha}+3e^{-\frac{(\lambda_1-\kappa) t_n}{4}}$ and
\begin{align*}
C_8:=\Big\{\begin{array}{cc}
\frac{1}{1-\alpha}(2^{1-\alpha}(1-e^{-\frac{(\lambda_1-\kappa) t_n}{4}})-1), & t_n >-\frac4{\lambda_1-\kappa}\ln(1-2^{\alpha-1}),\\
0, &t_n \leq-\frac4{\lambda_1-\kappa}\ln(1-2^{\alpha-1}).
\end{array}
\end{align*}
\end{theorem}

\subsection{Decaying error estimate of the C-N semi-discrete scheme}
The C-N semi-discrete scheme is given by
\begin{equation}\label{EQ_pdeCN}
\begin{aligned}
\frac{1}{\tau}\nabla_\tau U^n &= \Delta U^{n-\frac1{2}}+\kappa U^{n-\frac12}+f(t_{n-\frac12}),&& x\in\Omega,\;\;\;1\leq n\leq N,\\
U^0(x) &= u_0(x),&&x\in \Omega,\\
U^n(x)&=0,&&x\in\partial \Omega, \;1\leq n\leq N.
\end{aligned}
\end{equation}
The error $e^n$ satisfies
\begin{equation}\label{EQ_0100}
\begin{aligned}
\frac{1}{\tau} \nabla_\tau e^n &= \Delta e^{n-\frac1{2}}+\kappa e^{n-\frac12} +R_1^n+R_2^n,&& x\in\Omega,\;\;\;1\leq n\leq N,\\
e^0(x) &= u_0(x),&&x\in \Omega,\\
e^n(x)&=0,&&x\in \partial \Omega,\;1\leq n\leq N,
\end{aligned}
\end{equation}
where
\begin{align}
R_1^n &:= \frac{u^n-u^{n-1}}{\tau}-u'(t_{n-\frac12})=\frac1{2\tau}\big(\int_{t_{n-\frac12}}^{t_n}(t_n-s)^2\partial_{ttt}u(s)\d s+\int^{t_{n-\frac12}}_{t_{n-1}}(s-t_{n-1})^2 \partial_{ttt}u(s)\d s\big)\nonumber\\
&\leq \frac1{2\tau}\big(\int_{t_{n-\frac12}}^{t_n}(t_n-s)^2\|\partial_{ttt}u(s)\|\d s+\int^{t_{n-\frac12}}_{t_{n-1}}(s-t_{n-1})^2 \|\partial_{ttt}u(s)\|\d s\big):=\hat{R}^n_1,\label{EQ_2231}
\end{align}
and
\begin{align}
R_2^n &:= \frac{u^n+u^{n-1}}{2}-u(t_{n-\frac12}) =\frac1{2}\big(\int_{t_{n-\frac12}}^{t_n}(t_n-s)\partial_{tt}u(s)\d s+\int^{t_{n-\frac12}}_{t_{n-1}}(s-t_{n-1}) \partial_{tt}u(s)\d s\big)\nonumber\\
&\leq \frac1{2}\big(\int_{t_{n-\frac12}}^{t_n}(t_n-s)\|\partial_{tt}u(s)\|\d s+\int^{t_{n-\frac12}}_{t_{n-1}}(s-t_{n-1}) \|\partial_{tt}u(s)\|\d s\big):=\hat{R}^n_2. \label{EQ_2232}
\end{align}
Combining with the assumption \eqref{EQ_Fregularity}, one has
\begin{align}
R^n\leq \hat{R}^n\leq \Big\{\begin{array}{cc}C_{u,\alpha}\tau^\alpha,& k = 1,\\
C_{u,\alpha} \tau^2 t_{k-1}^{\alpha-3},& k\geq 2,\label{EQ_2212}
\end{array}
\end{align}
where $R^n:=R_1^n+R_2^n$ and $\hat{R}^n:=\hat{R}_1^n+\hat{R}_2^n$
\begin{theorem}\label{TH_CNpde}
Let $u^n$ {and} $U^n$ be the solutions to \eqref{EQ_exact} and \eqref{EQ_pdeCN}{,}  respectively.  Let $\lambda_1$ be the minimal eigenvalue of {the}  eigenvalue problem \eqref{EQ_eigen}{, and let $\kappa<\lambda_1$. Then,} for $ \tau <1/(\lambda_1-\kappa)$, it holds that
\begin{align}
\|e^n\|\leq e^{-(\lambda_1-\kappa) t_n}\|e^{0}\|+ C_{u,\alpha} \Big(C_9e^{-\frac{(\lambda_1-\kappa) t_n}{2}}\tau^\alpha+C_{10}t_n^{\alpha-2}\tau^2\Big),\label{EQ_0114}
\end{align}
where $C_9:=\frac{9}{2-\alpha}+12e^{-\frac{(\lambda_1-\kappa) t_n}{2}}$ and
\begin{align*}
C_{10}:=\Big\{\begin{array}{cc} \frac1{2-\alpha}(2^{2-\alpha}(1-e^{-\frac{7(\lambda_1-\kappa) t_n}{12}})-1),&t_n> -\frac{12}{7(\lambda_1-\kappa)}\ln(1-2^{\alpha-2}),\\
0,&t_n\leq  -\frac{12}{7(\lambda_1-\kappa)}\ln(1-2^{\alpha-2}).
\end{array}
\end{align*}
\end{theorem}

\subsection{Decay-preserving error estimate of the BDF2 semi-discrete scheme}
The BDF2 semi-discrete scheme is given by
\begin{equation}\label{EQ_pdeBDF2}
\begin{aligned}
U^1-U^{0}&=\tau \Delta U^{1}+\kappa \tau U^1+\tau f^{1},&&x\in \Omega,\;\\
\frac32U^n-2U^{n-1}+\frac12U^{n-2}&=\tau \Delta U^{n}+\kappa\tau U^n+\tau f^{n}, &&x\in \Omega,\;\;\; 2\leq n\leq N\\
U^0(x) &= u_0(x),&&x\in \Omega\\
U^n(x)& = 0,&&x\in \partial \Omega,\; 1\leq n\leq N.
\end{aligned}
\end{equation}
Let the global error $e^n:=u^n-U^n(0\leq n\leq N)$. Hence, the global error $e^n$ solves
\begin{align*}
e^1-e^{0}&=\tau \Delta e^{1}+\kappa\tau e^1+\tau \eta^{1},&&x\in \Omega,\;\\
\frac32e^n-2e^{n-1}+\frac12e^{n-2}&=\tau \Delta e^{n}+\kappa\tau e^n+\tau \eta^{n}, &&x\in \Omega,\;\;\; 2\leq n\leq N\\
e^0(x) &= 0,&&x\in \Omega\\
e^n(x)&=0,&&x\in \partial \Omega, \; 1\leq n\leq N,
\end{align*}
where $\eta^n$ denotes the truncation error.
\begin{theorem}\label{TH_BDF2pde}
Let $u^n$ {and} $U^n$ be the solutions to \eqref{EQ_exact} and \eqref{EQ_pdeBDF2}{,}  respectively.  Let $\lambda_1$ be the minimal eigenvalue of {the} eigenvalue problem \eqref{EQ_eigen}{,  and let $\kappa<\lambda_1$. Then,} for $ \tau <1/(4\lambda_1-\kappa)$, it holds that
\begin{align}
\|e^n\|\leq e^{-\frac{(\lambda_1-\kappa) t_n}{2}}\|e^{0}\|+ 2C_{u,\alpha} \Big(C_{11}e^{-\frac{(\lambda_1-\kappa) t_n}{2}}\tau^\alpha+C_{12}t_n^{\alpha-2}\tau^2\Big),\label{EQ_0026}
\end{align}
where $C_{11}:=\frac{2}{2-\alpha}+4e^{-\frac{(\lambda_1-\kappa) t_n}{4}}$ and
\begin{align*}
C_{12}:=\Big\{\begin{array}{cc} \frac1{2-\alpha}(2^{2-\alpha}(1-e^{-\frac{(\lambda_1-\kappa) t_n}{2}})-1),& t_n> -\frac{2}{\lambda_1-\kappa}\ln(1-2^{\alpha-2}),\\
0, & t_n\leq  -\frac{2}{\lambda_1-\kappa}\ln(1-2^{\alpha-2}).
\end{array}
\end{align*}
\end{theorem}

\section{The link between convergence regimes and model parameters \label{Sec_4}}
{With the} decay-preserving  error {estimates established, we} now illustrate {how these estimates depend} on model parameters (i.e., the reaction coefficient $\kappa$, final time $T$, and spatial domain $\Omega$). To {achieve this}, we apply our error estimates to investigate and { predict various} convergence regimes, ranging from { lower-order to high-order, at} the last time level { for different choices of} model parameters. {Initially, we demonstrate} the effectiveness of error estimates {in addressing} various convergence regimes for ODEs and diffusion equations {separately.  Subsequently, we  propose} a  conjecture for decay-preserving error estimates of subdiffusion equations \eqref{EQ_FPDEs} and provide numerical {evidence} to {validate}  the conjecture.
\subsection{Discussion and numerical experiments for ODEs  \eqref{EQ_ordinary} and PDEs \eqref{EQ_exact}}
We now {examine} the relationship between convergence rates and model parameters for solving ODEs \eqref{EQ_ordinary} and PDEs \eqref{EQ_exact}.
Based on Theorems \ref{TH_IEode}, \ref{TH_CNode}, \ref{TH_BDF2ode} in Section \ref{Sec_2}{,}  and  Theorems \ref{TH_IEpde}, \ref{TH_CNpde}, \ref{TH_BDF2pde} in Section \ref{Sec_3}{, while setting} $e^0 = 0$,  the error estimates  at {the}  final time level $t_N = N\tau (= T)$ for IE, C-N and BDF2 schemes can be {expressed in}  the following unified form
\begin{align}
|e^N| &\leq C(e^{C\kappa T}\tau^{\alpha}+ T^{\alpha-k}\tau^k ), && \text{for ODEs}\label{e1}\\
\|e^N\|& \leq C(e^{-C(\lambda_1-\kappa) T}\tau^{\alpha}+ T^{\alpha-k}\tau^k ),&&\text{for PDEs}\label{e2}
\end{align}
where $k = 1$ indicates implicit  Euler scheme and $k=2$ indicates C-N or BDF2 scheme. Here $\lambda_1$  is the minimal eigenvalue of {the} eigenvalue problem \eqref{EQ_eigen}, which is determined by the spatial domain $\Omega$.

{Observing} that \eqref{e1} can be {regarded}  as a special case ($\lambda_1=0$) of \eqref{e2}{, we can further represent} \eqref{e1} and \eqref{e2} {in}  the following unified form
\begin{align}
\|e^N\|& \leq C(e^{-C(\lambda_1-\kappa) T}\tau^{\alpha}+ T^{\alpha-k}\tau^k )\label{e12}
\end{align}
with $\lambda_1>0$ for PDEs \eqref{EQ_exact} and $\lambda_1\equiv0$ for ODEs \eqref{EQ_ordinary}.

A fundamental question in \eqref{e12} is {how the} error estimates reflect and reproduce  different convergence rate regimes{,}  ranging from $\alpha$ to $k$ order{, by considering} different model parameters.  From a theoretical  point of view,  $\tau^\alpha$ is the leading order as  $\tau\to 0$. However, due  to limited computational resources,  we can only take a finite $N$ in practice. Once $N$ is finite (i.e., $\tau$ is not infinitesimal), the coefficients  of $\tau^\alpha$ and $\tau^k$ in \eqref{e12}  {become crucial in determining} which is dominant. Noting that the coefficient of $\tau^\alpha$ {takes the form of an} exponential function containing parameters $T, \kappa,\lambda_1$. As we  know, exponential decay is much  faster {than algebraic decay,} and thus {these}  two coefficients {represent}  two scales for some $\kappa, \lambda_1${,}  and $T$. {Therefore}, one reasonable explanation {for the}  different convergence rate regimes { arises} from  the trade-off or competition between the two scales in \eqref{e12}, i.e., $\mathcal{O}(e^{C\kappa T}\tau^{\alpha})$ and $\mathcal{O}(T^{\alpha-k}\tau^k)$.  
We present the qualitative analysis from \eqref{e3} as follows.

\begin{itemize}
  \item {\bf Case 1:} $\kappa\geq \lambda_1.$ In this situation, the coefficient $e^{-C(\lambda_1-\kappa) T}$ will { exponentially increase, making} the first term in \eqref{e12}, i.e.,  $e^{-C(\lambda_1-\kappa)  T}\tau^{\alpha}$, {consistently} dominant. Hence, for $\kappa\geq \lambda_1$, the convergence rate always {behaves} as $\alpha$-order.

\item {\bf Case 2:} $\kappa< \lambda_1.$ In this situation, the coefficient $e^{-C(\lambda_1-\kappa)  T}$ {undergoes exponential decay}.  In practical simulations,  there {exists}  a lower bound $\tau_0${, ensuring that all choices of the} time step $\tau$ {fall within} the interval $[\tau_0,\infty)$. {For $\tau\geq\tau_0$, sufficiently large $T$, and  sufficiently large $\lambda_1-\kappa$} (relative small $\kappa$ or large $\lambda_1$),  we have  $e^{-C(\lambda_1-\kappa) T}\tau^\alpha  \ll T^{\alpha-k}\tau^k${. This leads to the simplification of} our error estimate \eqref{e12}  to $\|e^N\| \leq C\tau^{k},\;\forall \tau \geq  \tau_0$.  Actually, fixing $\kappa,\lambda_1$ and $T$ in this situation, there is a $\tau_1\ll \tau_0$ such that $(e^{-C(\lambda_1-\kappa) T}\tau^\alpha)/  (T^{\alpha-k}\tau^k) = \mathcal{O}(1)$ for $\tau\leq \tau_1$ since $\tau^\alpha$ decreases much slower than $\tau^k$. {Consequently,} our error estimate \eqref{e12}  can be {simplified} to $\|e^N\| \leq C\tau^{\alpha},\;\forall \tau \leq  \tau_1$. Similarly, for sufficient small $T$ {and} $\lambda_1-\kappa$ (relative large $\kappa$ or small $\lambda_1$), {the} coefficients $e^{-C(\lambda_1-\kappa) T}$ and $T^{\alpha-k}$ are {on}  the same scale{, allowing us to reduce} our error estimate \eqref{e2} to $\|e^N\| \leq C\tau^{\alpha},\;\forall \tau >0$. {Therefore, by considering} the model parameters $\kappa, \lambda_1$, and $T$, we can determine the sizes of $\tau_0$ and $\tau_1$ to capture different convergence regimes  in practical simulations.

  In summary, we can qualitatively illustrate the convergence order based on the values of model parameters (i.e., $\kappa, \lambda_1$ and $T$) as follows:
 \begin{itemize}
 \item Given a lower bound of time steps $\tau_0$ and letting $\tau\geq \tau_0$, then
   \begin{enumerate}
    \item[(C1)] for sufficiently small $\lambda_1-\kappa$ (relative large $\kappa$ or small $\lambda_1$) and $T$,  the convergence order is $\mathcal{O}(\tau^\alpha)$;
    \item[(C2)]  for sufficiently large $\lambda_1-\kappa$ (relative small $\kappa$ or large $\lambda_1$) and $T$, the convergence order is $\mathcal{O}(\tau^k)$.
  \end{enumerate}
  \item Given the model parameters $\kappa, \lambda_1$ and $T$, then
     \begin{enumerate}
    \item[(C3)] for  sufficiently small $\tau$,  the convergence order is $\mathcal{O}(\tau^\alpha)$.
  \end{enumerate}
  \end{itemize}
\end{itemize}

 \begin{table}[!ht]
\begin{center}
\caption{({\bf Example 4.1} for ODEs) The convergence rates  with $T=1$  by taking various $\kappa$.
\label{T_ODEi} } \vspace*{0.5pt}
\def\temptablewidth{1.0\textwidth}
{\rule{\temptablewidth}{1pt}}
\begin{tabular*}{\temptablewidth}{@{\extracolsep{\fill}}ccccccc}
&$N$& $\kappa = -1$&$\kappa = -5$ & $\kappa = -10$&$\kappa = -15$&$\kappa = -20$\\
\hline
     &256 &0.50&0.71 &1.00&1.00&1.00\\
  IE&512 &0.50&0.66 &0.99 &1.00&1.00\\
     &1024 &0.50&0.62 &0.98 &1.00&1.00\\
     &2048 &0.50&0.59 &0.96 &1.00&1.00\\
\hline
        &128  &0.50 &0.44  &2.83&2.00 &2.00\\
        &256 &0.50&0.48 &1.73&2.01&2.00\\
C-N &512 &0.50&0.49 &-0.63&2.03 &2.00\\
        &1024 &0.50&0.50 &0.25 &2.09&2.00\\
        &2048 &0.50&0.50 &0.42 &2.29&2.00\\
     \hline
             &128  &0.50 &0.44  &1.57 &2.04&2.02\\
            &256 &0.50&0.48 &-0.55&2.07&2.01\\
  BDF2 &512 &0.50&0.49 &0.26 &2.18&2.01\\
             &1024 &0.50&0.50 &0.42 &2.65&2.01\\
             &2048 &0.50&0.50 &0.47 &2.71&2.01\\

\end{tabular*}
{\rule{\temptablewidth}{1pt}}
\end{center}
\end{table}

\begin{table}[!ht]
\begin{center}
\caption{({\bf Example 4.1} for ODEs) The  convergence rates  with $\kappa=-1$  by taking various $T$.
\label{T_ODEii} } \vspace*{0.5pt}
\def\temptablewidth{1.0\textwidth}
{\rule{\temptablewidth}{1pt}}
\begin{tabular*}{\temptablewidth}{@{\extracolsep{\fill}}ccccccc}
&$N$& $T= 1$&$T = 5$ & $T = 10$&$T = 15$&$T = 20$\\
\hline
     &256 &0.50&0.71 &1.00&1.00&1.00\\
  IE&512 &0.50&0.66 &0.99 &1.00&1.00\\
     &1024 &0.50&0.62 &0.98 &1.00&1.00\\
     &2048 &0.50&0.59 &0.96 &1.00&1.00\\
\hline
        &128  &0.50 &0.44  &2.83&2.00 &2.00\\
        &256 &0.50&0.48 &1.73&2.01&2.00\\
C-N &512 &0.50&0.49 &-0.63&2.03 &2.00\\
        &1024 &0.50&0.50 &0.25 &2.09&2.00\\
        &2048 &0.50&0.50 &0.42 &2.29&2.00\\
     \hline
             &128  &0.50 &0.44  &1.57 &2.04&2.02\\
            &256 &0.50&0.48 &-0.55&2.07&2.01\\
  BDF2 &512 &0.50&0.49 &0.26 &2.18&2.01\\
             &1024 &0.50&0.50 &0.42 &2.65&2.01\\
             &2048 &0.50&0.50 &0.47 &2.71&2.01\\

\end{tabular*}
{\rule{\temptablewidth}{1pt}}
\end{center}
\end{table}

\begin{table}[!ht]
\begin{center}
\caption{({\bf Example 4.1} for ODEs) The convergence rates  with various $\kappa\geq0$ and $T$.
\label{T_ODEiii} } \vspace*{0.5pt}
\def\temptablewidth{1.0\textwidth}
{\rule{\temptablewidth}{1pt}}
\begin{tabular*}{\temptablewidth}{@{\extracolsep{\fill}}cccccc}
&$N$ &\multicolumn{2}{c}{$\kappa=0$}
&\multicolumn{2}{c}{$\kappa = 0.5$} \\
\cline{3-4}    \cline{5-6}
&& $T= 1$&$T = 5$ & $T = 1$&$T = 5$\\
\hline

     &256 &0.49&0.49 &0.48&0.49\\
  IE&512 &0.49&0.49 &0.49 &0.49\\
     &1024 &0.49&0.49 &0.49 &0.49\\
     &2048 &0.50&0.50 &0.49 &0.49\\
\hline
        &256 &0.50&0.50 &0.50&0.50\\
C-N &512 &0.50&0.50 &0.50&0.50\\
        &1024 &0.50&0.50 &0.50 &0.50\\
        &2048 &0.50&0.50 &0.50 &0.50\\
     \hline
        &256 &0.50&0.50 &0.50&0.50\\
BDF2 &512 &0.50&0.50 &0.50&0.50\\
        &1024 &0.50&0.50 &0.50 &0.50\\
        &2048 &0.50&0.50 &0.50 &0.50\\

\end{tabular*}
{\rule{\temptablewidth}{1pt}}
\end{center}
\end{table}


\begin{example}\label{Ex1}
In this example, we shall illustrate various convergence rates in different model parameter regimes for solving ODEs \eqref{EQ_ordinary} (i.e., the case of $\lambda_1=0$).
To {achieve this}, we construct a benchmark solution {for}  ODE \eqref{EQ_ordinary} as $ u =10+t^\alpha$. The convergence order   at time level $t_N = T$ is usually  calculated by
\begin{align}\text{Order}(N)= \log_2(|e^{N/2}|/|e^N|).\label{EQ_calorder1}\end{align}

Table \ref{T_ODEi} {presents} convergence rates for the three numerical schemes{, with fixed} $T=1$ and increasing  $|\kappa|$ and $N$. Table \ref{T_ODEii} {displays} convergence rates {with fixed}  $\kappa=-1$ and increasing  $T$ and $N$. {In} Table \ref{T_ODEiii}{,  convergence rates are shown for} $\kappa = 0,1$, and $T = 1,5${, with $N$ increasing}.
{Analyzing} Tables \ref{T_ODEi} and \ref{T_ODEii} {reveals that}  the convergence rates {change} from $\mathcal{O}(\tau^\alpha)$ to $\mathcal{O}(\tau^k)$ ($k=1,2$)  as  $|\kappa|$ and $T$ increase, {aligning}  with our theoretical analysis in {\bf Case 2} and conclusions C1, C2.  {For cases such as}  $\kappa=-5$ or $T=5$ for IE scheme and $\kappa=-10$ or $T=10$ for C-N {and} BDF2 schemes{, as shown} in Tables \ref{T_ODEi} and \ref{T_ODEii},  the convergence rates tend {towards} $\mathcal{O}(\tau^\alpha)$  as  $N$ increases{. This observation is in line} with our theoretical analysis in {\bf Case 2} and conclusions C3. {In}  Table \ref{T_ODEiii}{, the convergence rate consistently remain} $\mathcal{O}(\tau^\alpha)$ as long as $\kappa\geq 0$,  which is consistent with our theoretical analysis in {\bf Case 1}  for $\kappa\geq 0$.

\end{example}
\begin{example}\label{Ex2}
In this example, we {aim to} illustrate various convergence rates in different model parameter regimes for PDEs \eqref{EQ_exact} (i.e., the case of $\lambda_1>0$).
 To facilitate the immediate calculation of the minimal eigenvalue $\lambda_1$ of {the} eigenvalue problem \eqref{EQ_eigen}, we choose the spatial domain $\Omega = (0,L)$. {It's important to note that the observed} numerical behaviors of convergence {rates} in {\bf Case 1} and {\bf Case 2}  are based on temporal semi-discretizations. {This} implies that the choices of numerical methods for spatial discretization would not {introduce additional} influence on the convergence rates. For simplicity, we employ the finite difference method  for spatial discretization, where the  spatial length $h = L/M${, with $M$ being a positive integer}, and the  discrete $L^2$-norm  is denoted by $\|\cdot\|_h$. {Further details on spatial discretization can be found in} \cite{LLZ18}.


We also use the benchmark solution of  PDEs \eqref{EQ_exact} constructed in \eqref{exact}, i.e., $ u = t^{\alpha}\sin(\pi x/L)$.
Table \ref{T_PDEi} {presents} convergence rates of the three numerical schemes{, with fixed} $T=1,L=\pi$, and decreasing  $\kappa$ {while} increasing $N$. Table \ref{T_PDEi2} {illustrates} convergence rates {with fixed} $\kappa=0, T= 10$, and  {an increase in} spatial length $L$ and $N$. Table \ref{T_PDEii} {displays}  convergence rates with $\kappa = 0, L = \pi$, and  {and an increase in the}  final time $T$. {Lastly,} Table \ref{T_PDEiii} shows  convergence rates with $T=5$, $\kappa = 1,1.5, L = \pi,4$ and   increasing  $N$.

 As {observed} in Tables \ref{T_PDEi}, \ref{T_PDEi2} and \ref{T_PDEii},  the convergence rates  {change} from $\mathcal{O}(\tau^\alpha)$ to $\mathcal{O}(\tau^k)$ ($k=1,2$)  as  $\lambda_1$ and $T$ increase or $\kappa$ decreases,{. This observation aligns} with our theoretical analysis in {\bf Case 2} and conclusions C1,C2.
 In {instances where}  $\kappa=-5$ or $T=5$ or $L=4,5$ for IE scheme and $\kappa=-10$ or $L=3$ or $T=10$ for C-N {and} BDF2 schemes{, as shown} in Tables \ref{T_PDEi}, \ref{T_PDEi2} and \ref{T_PDEii},  the convergence rates tend { towards} $\mathcal{O}(\tau^\alpha)$  as  $N$ increases, {consistent} with {\bf Case 2} and conclusions C3.
   Table \ref{T_PDEiii} {illustrates that} the convergence rate {remains} $\mathcal{O}(\tau^\alpha)$ as long as $\kappa\geq \lambda_1$,  which is consistent with {\bf Case 1}  for $\kappa\geq \lambda_1$.
\end{example}

Another  interesting observation  is that  these schemes {exhibit} different numerical behaviors when convergence rates  change from $k$ to $\alpha$. {Specifically, Tables}  \ref{T_ODEi} and \ref{T_ODEii} for ODEs \eqref{EQ_ordinary} and Tables  \ref{T_PDEi}, \ref{T_PDEi2} and \ref{T_PDEii} for PDEs \eqref{EQ_exact} further {demonstrate}

\begin{enumerate}
   \item [(P1)] the convergence rates of {the} implicit Euler scheme change  monotonously with respect to $N$ and {the} model parameters $\kappa,T,\lambda_1$;
   \item [(P2)] some {unusual orders,} such as 2.68, 0.29 even -0.55 appear for BDF2  and C-N schemes.
\end{enumerate}

 The numerical phenomenon (P1) {can also} be well interpreted by our decay-preserving estimates.
Inserting the estimate \eqref{e12} into the formula $\text{Order}(N):= \log_2(\|e^{N/2}\|)/\|e^N\|)$, one has
\begin{align}
\text{Order($N$)}&=\log_2(\frac{\|e^{N/2}\|}{\|e^{N}\|})=\log_2\big(\frac{2^{\alpha}e^{-C (\lambda_1-\kappa) T }\tau^{\alpha}+2^kT^{\alpha-k}\tau^k}{e^{-C(\lambda_1-\kappa)  T}\tau^{\alpha}+T^{\alpha-k} \tau^k}\big)\nonumber\\
&=\log_2\Big(\frac{2^k}{1+\frac{2^{k-\alpha}-1}{1+2^{k-\alpha}e^{(\lambda_1-\kappa)  T}/N^{k-\alpha}}}\Big).\label{Porder}
\end{align}
From \eqref{Porder},  the convergence rate {monotonously increases} with respect to $\lambda_1$ and $T${, transitioning} from $\alpha$-order to $k$-order{,  and simultaneously  decreases} with respect to $\kappa$ {and}  $N${. This observation aligns} with numerical phenomenon (P1){, suggesting that} our decay-preserving error estimates   are { sufficiently sharp} to capture the {entire progression of convergence orders in the} implicit Euler scheme as $T, \lambda_1, \kappa$, and $N$ change.

To {better illustrate}  numerical phenomenon (P2),   Figures \ref{F1a} and \ref{F1b}  {depict} the absolute  {errors}  for  ODEs \eqref{EQ_ordinary} with fixed model parameters $T=14$ {and}  $\kappa = -1${. Additionally,} Figures \ref{F1c} and \ref{F1d}  {displays} the $L^2$-norm { errors} for PDEs \eqref{EQ_exact}  with fixed model parameters $T=12, \kappa = 0$, {and} $L = \pi (\lambda_1-\kappa = 1)$.
As {observed, Figure \ref{Fig_PDEtick} exhibits kinks, suggesting the presence of a} sup-convergence point $N_0$ {where} the absolute ($L^2$-norm) {errors become notably small. Unfortunately,} our current analysis cannot {elusidate this} sup-convergence phenomenon, which may {necessitate a}  refined error estimate in the future.


 \begin{table}[!ht]
\begin{center}
\caption {({\bf Example 4.2} for PDEs) The convergence rates  with $T=1, L = \pi$  by taking various $\kappa$.
\label{T_PDEi} } \vspace*{0.5pt}
\def\temptablewidth{1.0\textwidth}
{\rule{\temptablewidth}{1pt}}
\begin{tabular*}{\temptablewidth}{@{\extracolsep{\fill}}ccccccc}
&$N$& $\kappa = 0$& $\kappa = -5$&$\kappa = -10$ & $\kappa = -15$&$\kappa = -20$\\
\hline
     &256 &0.50&0.81 &1.00&1.00&1.00\\
  IE&512 &0.50&0.75 &1.00 &1.00&1.00\\
     &1024 &0.50&0.70 &0.99 &1.00&1.00\\
     &2048 &0.50&0.66 &0.98 &1.00&1.00\\
\hline
        &256 &0.50&0.45 &2.03&2.00&2.00\\
C-N &512 &0.50&0.48 &1.36&2.01 &2.00\\
        &1024 &0.50&0.49 &-0.50 &2.03&2.00\\
        &2048 &0.50&0.50 &0.26 &2.10&2.00\\
     \hline
            &256 &0.50&0.46 &0.23&2.03&2.01\\
  BDF2 &512 &0.50&0.48 &-0.22 &2.08&2.01\\
             &1024 &0.50&0.49 &0.31 &2.24&2.04\\
             &2048 &0.50&0.50 &0.44 &3.01&2.08\\

\end{tabular*}
{\rule{\temptablewidth}{1pt}}
\end{center}
\end{table}

\begin{table}[!ht]
\begin{center}
\caption {({\bf Example 4.2} for PDEs) The convergence rates  with $\kappa=0, T = 10$  by taking various $L$.
\label{T_PDEi2} } \vspace*{0.5pt}
\def\temptablewidth{1.0\textwidth}
{\rule{\temptablewidth}{1pt}}
\begin{tabular*}{\temptablewidth}{@{\extracolsep{\fill}}ccccccc}
&$N$ &\multicolumn{1}{c}{$L=1$} &\multicolumn{1}{c}{$L=2$} &\multicolumn{1}{c}{$L=3$} &\multicolumn{1}{c}{$L=4$} &\multicolumn{1}{c}{$L=5$}
\\
&& $(\lambda_1 \approx9.87 )$&$(\lambda_1 \approx 2.47)$ & $(\lambda_1 \approx 1.10)$&$(\lambda_1 \approx 0.62)$&$(\lambda_1 \approx 0.39)$\\
\hline
     &256 &1.00&1.00 &1.00&0.82&0.63\\
  IE&512 &1.00&1.00 &1.00 &0.77&0.59\\
     &1024 &1.01&1.00 &0.99 &0.72&0.57\\
     &2048 &1.01&1.00 &0.98 &0.67&0.55\\
\hline
        &256 &2.00&2.00 &2.98&0.44&0.49\\
C-N &512 &1.98&1.99 &1.18&0.48 &0.49\\
        &1024 &1.93&1.99 &-0.45 &0.49&0.50\\
        &2048 &1.86&1.98 &0.27 &0.50&0.50\\
     \hline
            &256 &1.67&2.04 &0.03&0.46&0.49\\
  BDF2 &512 &1.05&2.26 &-0.15 &0.48&0.49\\
             &1024 &3.77&1.28 &0.33 &0.49&0.50\\
             &2048 &-1.94&0.64 &0.39 &0.50&0.50\\
\end{tabular*}
{\rule{\temptablewidth}{1pt}}
\end{center}
\end{table}

\begin{table}[!ht]
\begin{center}
\caption {({\bf Example 4.2} for PDEs) The convergence rates  with $\kappa=0,L=\pi$  by taking various $T$.
\label{T_PDEii} } \vspace*{0.5pt}
\def\temptablewidth{1.0\textwidth}
{\rule{\temptablewidth}{1pt}}
\begin{tabular*}{\temptablewidth}{@{\extracolsep{\fill}}ccccccc}
&$N$& $T= 1$&$T = 5$ & $T = 10$&$T = 15$&$T = 20$\\
\hline

     &256 &0.50&0.71 &1.00&1.00&1.00\\
  IE&512 &0.50&0.66 &0.99 &1.00&1.00\\
     &1024 &0.50&0.62 &0.97 &1.00&1.00\\
     &2048 &0.50&0.59 &0.96 &1.00&1.00\\
\hline
        &256 &0.50&0.48 &1.74&2.01&2.00\\
C-N &512 &0.50&0.49 &-0.63&2.03 &2.00\\
        &1024 &0.50&0.50 &0.24 &2.10&2.01\\
        &2048 &0.50&0.50 &0.42 &2.37&2.07\\
     \hline
            &256 &0.50&0.48 &-0.56&2.01&2.05\\
  BDF2 &512 &0.50&0.49 &0.27 &2.39&2.31\\
             &1024 &0.50&0.50 &0.43 &3.53&1.37\\
             &2048 &0.50&0.50 &0.45 &1.79&0.75\\

\end{tabular*}
{\rule{\temptablewidth}{1pt}}
\end{center}
\end{table}

\begin{table}[!ht]
\begin{center}
\caption {({\bf Example 4.2} for PDEs) The convergence rates  with $T = 5$ and different $\kappa\geq\lambda_1$ and $L$.
\label{T_PDEiii} } \vspace*{0.5pt}
\def\temptablewidth{1.0\textwidth}
{\rule{\temptablewidth}{1pt}}
\begin{tabular*}{\temptablewidth}{@{\extracolsep{\fill}}cccccc}
&$N$ &\multicolumn{2}{c}{$\kappa=1$}
&\multicolumn{2}{c}{$\kappa = 1.5$} \\
\cline{3-4}    \cline{5-6}
&& $L = \pi$&$L = 4$ & $L=\pi$&$L=4$\\
&&$(\lambda_1 = 1)$ & $(\lambda_1\approx 0.62)$&$(\lambda_1=1)$& $(\lambda_1\approx 0.62)$\\
\hline

     &256 &0.49&0.48 &0.49&0.52\\
  IE&512 &0.49&0.49 &0.49 &0.50\\
     &1024 &0.49&0.49 &0.49 &0.50\\
     &2048 &0.50&0.49 &0.49 &0.50\\
\hline
        &256 &0.50&0.50 &0.50&0.51\\
C-N &512 &0.50&0.50 &0.50&0.50\\
        &1024 &0.50&0.50 &0.50 &0.50\\
        &2048 &0.50&0.50 &0.50 &0.50\\
     \hline
        &256 &0.50&0.50 &0.50&0.51\\
BDF2 &512 &0.50&0.50 &0.50&0.50\\
        &1024 &0.50&0.50 &0.50 &0.50\\
        &2048 &0.50&0.50 &0.50 &0.50\\

\end{tabular*}
{\rule{\temptablewidth}{1pt}}
\end{center}
\end{table}

\begin{figure}[!ht]
\begin{center}
\subfigure[BDF2 scheme for ODEs\label{F1a}]{
\includegraphics[width=2.5in]{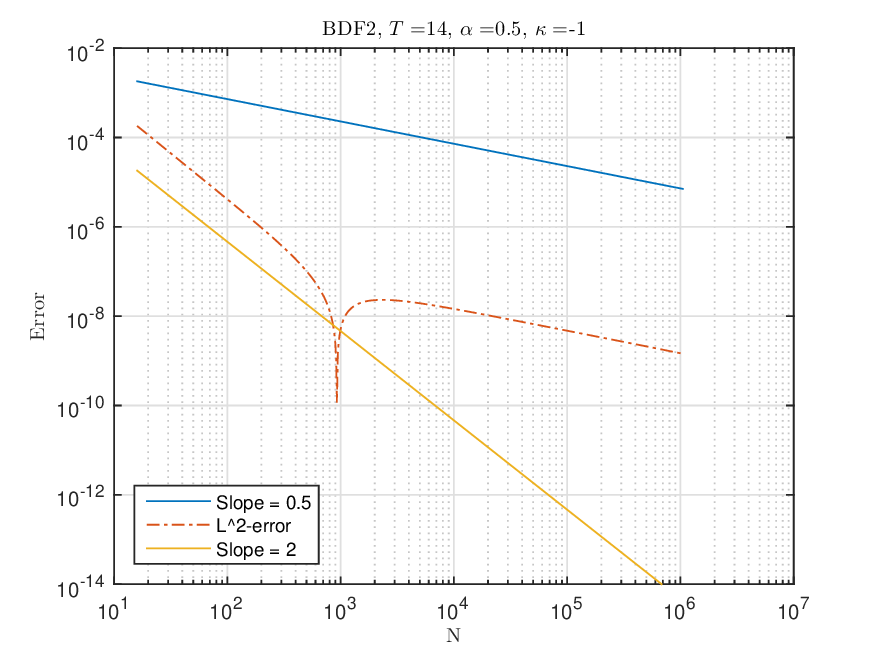}}
\subfigure[C-N scheme for ODEs\label{F1b}]{
\includegraphics[width=2.5in]{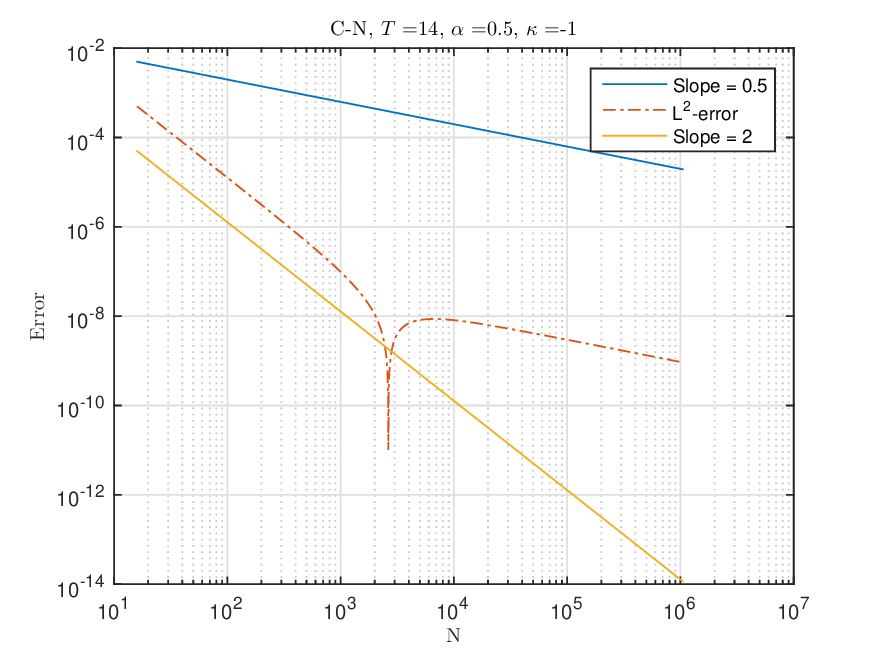}}\\
\subfigure[BDF2 scheme for PDEs\label{F1c}]{
\includegraphics[width=2.5in]{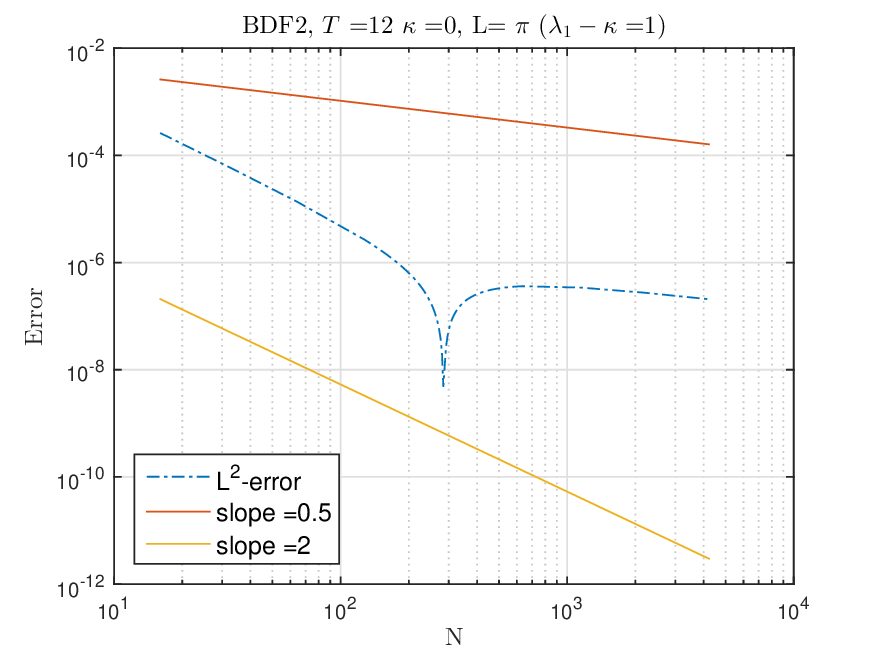}}
\subfigure[C-N scheme  for PDEs\label{F1d}]{
\includegraphics[width=2.5in]{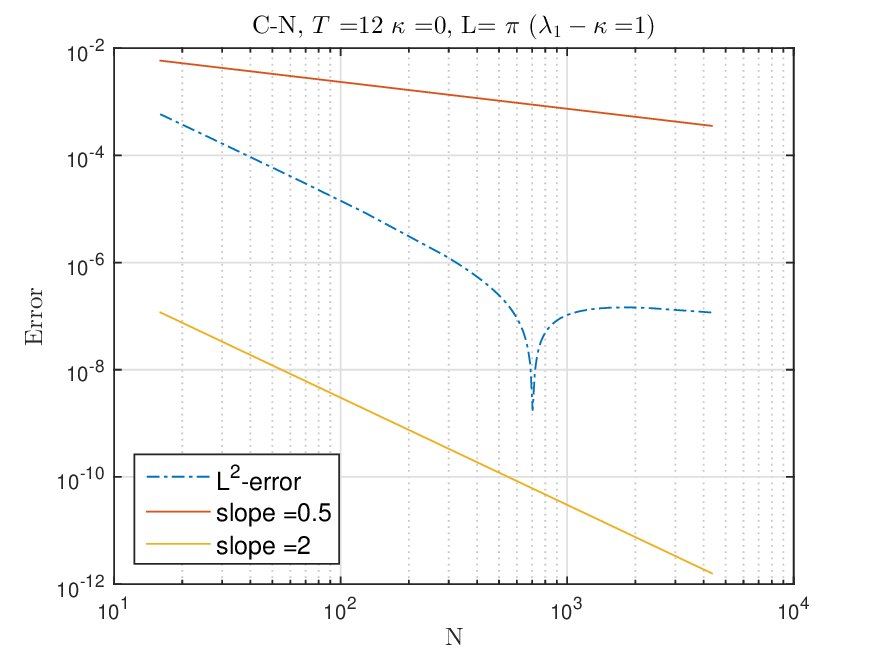}}
\caption{ Errors with fixed model parameters $T, \kappa, \lambda_1$: absolute error for ODEs and PDEs. \label{Fig_PDEtick}}
\end{center}
\end{figure}
\newpage

\subsection{Conjecture and numerical experiments for  sub-diffusion equations}
It is known that the solution to {the} sub-diffusion equation \eqref{EQ_FPDEs} {exhibits} weak regularity \eqref{EQ_Fregularity} even for smooth initial values. As {displayed} in Table \ref{T_F1}, {various}  convergence rates of the $L1$ scheme are observed. { Therefore, it is both significant and desirable  to formulate corresponding} decay-preserving estimates
    {for the} L1 scheme {applied to} equation \eqref{EQ_FPDEs}.
   
   { Let }$U^n$ be {the} numerical approximation to $u(x,t_n)$. The L1 formula \cite{LX07,SW06,LLZ18} {for approximating} the fractional Caputo derivative is defined {as follws:}
\begin{align}
\mathcal{D}_\tau^\alpha U^n : = \frac{\tau^{-\alpha}}{\Gamma(2-\alpha)}\sum_{j=1}^nA_{n-j}(U^j-U^{j-1}),
\end{align}
where $A_i = (i+1)^{1-\alpha}-i^{1-\alpha}$. Hence, the semi-discrete L1 scheme to \eqref{EQ_Fregularity} is given by
\begin{align}
\mathcal{D}_\tau^\alpha U^n  = \Delta U^n+\kappa U^n +f^n,\label{EQ_L1}
\end{align}
with $U^0(x) = u_0(x)$ and $U^n(x)=0,\;\forall x\in \partial \Omega$.

Due to the nonlocality of the  fractional operator,  {obtaining a} decay-preserving estimate {poses an essential challenge, as} the discrete fundamental solution to $\mathcal{D}_\tau^\alpha U^n = \kappa U^n$ is {difficult to explicitly formulate}.
{Therefore, we only propose} a convincing conjecture {regarding} the decay-preserving estimate of {the} L1 scheme for sub-diffusion equations \eqref{EQ_FPDEs} and {substantiate} our conjecture numerically. The following conjecture is {inspired}  by the error estimates for ODEs and PDEs and the fundamental solution of the following fractional ordinary equation
\begin{equation*}
  \partial_t^\alpha y(t) = \kappa y(t)+f(t), \quad y(0) = y_0,
\end{equation*}
  for example, see \cite[Remark 7.1]{D10}, which can be expressed in the form of
\begin{equation*}
  y(t) = y_0 E_\alpha(\kappa t^\alpha)+\alpha\int_0^t E_\alpha'(\kappa (t-s)^\alpha)(t-s)^{\alpha-1}f(s)\d s.
\end{equation*}
Here $E_\alpha(\cdot):=\sum_{j=0}^\infty \frac{(\cdot)^j}{\Gamma(j\alpha+1)}$ represents Mittag-Leffler function \cite{D10}.
The decay-preserving error estimate for {the}  $L1$ scheme \eqref{EQ_L1} is shown as follows.
    \begin{conjecture}\label{C_FPDEs}
Let $u^n$ {and} $U^n$ be the solutions to \eqref{EQ_FPDEs} and {the} $L1$ scheme \eqref{EQ_L1}, respectively. Let $\lambda_1$ be the minimal eigenvalue of {the} eigenvalue problem \eqref{EQ_eigen}, and  let $\kappa<\lambda_1$. Then for sufficiently small $\tau$, i{the following holds:}
\begin{align}
\|e^n\|\leq E_\alpha(-C(\lambda_1-\kappa) t_n^\alpha)\|e^{0}\|+ \Cu t_n^{\alpha-1}\Big( E_\alpha'(-C(\lambda_1-\kappa)  t_n^\alpha)\tau+\tau^{2-\alpha}\Big).\label{EQ_INfpde}
\end{align}
\end{conjecture}
\noindent Taking $e^0=0$ in \eqref{EQ_INfpde},  the error estimates at final time level $t_N = N\tau (= T)$ can be { expressed}  as
\begin{equation} \label{e3}
\|e^N\| \leq C(E_\alpha'(-C(\lambda_1-\kappa)  T^\alpha )\tau+ \tau^{2-\alpha} ).
\end{equation}
{It follows from \eqref{e3} that the coefficient $E_\alpha'(-C(\lambda_1-\kappa) T^\alpha)$ of $\tau$ is decreasing with respect to $T$ when $\lambda_1>\kappa$ and is increasing when  $\lambda_1\leq \kappa$.  As {discussed in the} previous subsection, the trade-off or competition between the two scales  $\mathcal{O}(E_\alpha'(-C(\lambda_1-\kappa)  T^\alpha)\tau^{\alpha})$ and $\mathcal{O}(\tau^{2-\alpha})$ in \eqref{e2} is the main cause of various convergence rate regimes. {Therefore},  we present { a}  {similar} qualitative analysis {as in}  \eqref{e3} as follows.}
\begin{itemize}
  \item {\bf Case 1:} $\kappa\geq \lambda_1.$ In this situation, the coefficient $E_\alpha'(-C(\lambda_1-\kappa)  T^\alpha)$ {increases} with respect to $T$, making the first term $E_\alpha'(-C(\lambda_1-\kappa)  T^\alpha) \tau$ in \eqref{e3} {consistently dominant. Consequently}, the convergence rate for $\kappa\geq \lambda_1$ always {behaves} as first-order, i.e., $\mathcal{O}(\tau)$.
\item {\bf Case 2:} $\kappa< \lambda_1.$ In this situation, the coefficient $E_\alpha'(-C(\lambda_1-\kappa)  T^\alpha)$  will decay with respect to $T$.
   {Consequently}, the convergence order  can be  qualitatively illustrated as follows:
 \begin{itemize}
 \item Given a lower bound  on the time steps $\tau_0$ and letting $\tau\geq \tau_0$, then
   \begin{enumerate}
    \item[(E1)] for {sufficiently} small $\lambda_1-\kappa$ ({relatively} large $\kappa$ or small $\lambda_1$) and $T$,  the convergence order is $\mathcal{O}(\tau)$;
    \item[(E2)]  for sufficiently large $\lambda_1-\kappa$ ({relatively} small $\kappa$ or large $\lambda_1$) and $T$, the convergence order is $\mathcal{O}(\tau^{2-\alpha})$.
  \end{enumerate}
  \item Given the model parameters $\kappa, \lambda_1$, and $T$, then
     \begin{enumerate}
    \item[(E3)] for {sufficiently} small $\tau$,  the convergence order is $\mathcal{O}(\tau)$.
  \end{enumerate}
  \end{itemize}
\end{itemize}
\begin{example}
{We illustrate} various convergence orders in different model parameter regimes {using the} $L1$ scheme \eqref{EQ_L1}. {We follow the setting of space  in Example \ref{Ex2}, including the notations and definitions}, and also use the benchmark solution of { the}  sub-diffusion equation  \eqref{EQ_FPDEs} constructed in \eqref{exact}, i.e., $ u = t^{\alpha}\sin(\pi x/L)$.
The convergence rates of {the} $L1$ scheme \eqref{EQ_L1} are shown in Table \ref{T_FPDEi}  by fixing $T=1,L=\pi$, and decreasing  $\kappa$,  in Table \ref{T_FPDEi2} by fixing $\kappa=0, T= 10$, and increasing $L$ and $N$, in Table \ref{T_FPDEii} by fixing $\kappa = 0, L = \pi$, and   increasing {the}  final time $T$, and in Table \ref{T_FPDEiii} by fixing $T=5$, $\kappa = 1,1.5, L = \pi,4$, and   increasing  $N$.

 As {observed}  in Tables \ref{T_FPDEi}, \ref{T_FPDEi2} and \ref{T_FPDEii},  the convergence rates {change} from $\mathcal{O}(\tau)$ to $\mathcal{O}(\tau^{2-\alpha})$   as  $\lambda_1$ and $T$ increase or $\kappa$  decreases, {consistent} with conclusions (E1) and (E2).  {Additionally},  the convergence rates {demonstrate a monotonous change, implying} our Conjecture \ref{C_FPDEs} is reasonable {and effective in revealing} the influence of the model parameters on {the}  convergence orders of {the}  $L1$ scheme.

 {For some choices of model parameters in Tables \ref{T_FPDEi}, \ref{T_FPDEi2} and \ref{T_FPDEii}, the convergence rates tend {towards}  $\mathcal{O}(\tau)$ as $N$ increases, {aligning} with conclusion (E3).} Table \ref{T_FPDEiii} {illustrates that} the convergence rate {remains}  $\mathcal{O}(\tau)$ as long as $\kappa\geq \lambda_1$,  {consistent} with our theoretical analysis in {\bf Case 1}  for $\kappa\geq \lambda_1$.

We point out that the decay rate of Mittag-Leffler function is less than {that of the} exponential function.
{Therefore, to attain}  the optimal convergence order,   the range of model parameters shall be chosen wider than {in the previous} subsection. As predicted, {a broader}  range of model parameters, such as $\kappa = -50, T = 100$, is selected in Tables \ref{T_FPDEi}, \ref{T_FPDEi2}, and \ref{T_FPDEii}, in comparison to model parameters $|\kappa|\leq 20, T\leq 20$ in Tables \ref{T_PDEi}, \ref{T_PDEi2}, and \ref{T_PDEii}.
\end{example}

\begin{table}[!ht]
\begin{center}
\caption {({\bf Example 4.3} for subdiffusion equation) Convergence rate with $T=1, L = \pi$ and various $\kappa$.
\label{T_FPDEi} } \vspace*{0.1pt}
\def\temptablewidth{1.0\textwidth}
{\rule{\temptablewidth}{1pt}}
\begin{tabular*}{\temptablewidth}{@{\extracolsep{\fill}}cccccc}
$N$& $\kappa = 0$& $\kappa = -5$&$\kappa = -10$ & $\kappa = -20$&$\kappa = -50$\\
\hline
     64 &1.06&1.23 &1.31&1.38&1.44\\
    128 &1.05&1.19 &1.26 &1.34&1.42\\
     256 &1.03&1.15 &1.22 &1.30&1.39\\
     512 &1.02&1.11 &1.18 &1.25&1.36\\
\end{tabular*}
{\rule{\temptablewidth}{1pt}}
\end{center}
\end{table}
\begin{table}[!ht]
\begin{center}
\caption {({\bf Example 4.3} for subdiffusion equation) Convergence rate with $\kappa=0, T = 10$ and various $L$.
\label{T_FPDEi2} } \vspace*{0.1pt}
\def\temptablewidth{1.0\textwidth}
{\rule{\temptablewidth}{1pt}}
\begin{tabular*}{\temptablewidth}{@{\extracolsep{\fill}}cccccc}
$N$ &\multicolumn{1}{c}{$L=1$} &\multicolumn{1}{c}{$L=2$} &\multicolumn{1}{c}{$L=3$} &\multicolumn{1}{c}{$L=4$} &\multicolumn{1}{c}{$L=5$}
\\
& $(\lambda_1 \approx9.87 )$&$(\lambda_1 \approx 2.47)$ & $(\lambda_1 \approx 1.10)$&$(\lambda_1 \approx 0.62)$&$(\lambda_1 \approx 0.39)$\\
\hline
     64 &1.41&1.26 &1.17&1.11&1.08\\
  128 &1.38&1.22&1.13 &1.08&1.06\\
     256 &1.34&1.18 &1.10 &1.06&1.04\\
     512 &1.31&1.14 &1.07 &1.04&1.03\\
\end{tabular*}
{\rule{\temptablewidth}{1pt}}
\end{center}
\end{table}
\begin{table}[!ht]
\begin{center}
\caption {({\bf Example 4.3} for subdiffusion equation) Convergence rate  with $\kappa=0,L=\pi$ and various $T$.
\label{T_FPDEii} } \vspace*{0.1pt}
\def\temptablewidth{1.0\textwidth}
{\rule{\temptablewidth}{1pt}}
\begin{tabular*}{\temptablewidth}{@{\extracolsep{\fill}}ccccccc}
$N$& $T= 1$&$T = 10$ & $T = 20$&$T = 50$&$T = 100$\\
\hline
     64 &1.06&1.16 &1.20&1.25&1.30\\
  128 &1.05&1.12&1.15 &1.21&1.25\\
     256 &1.03&1.09&1.12 &1.17&1.21\\
     512 &1.02&1.07 &1.09 &1.13&1.17\\
\end{tabular*}
{\rule{\temptablewidth}{1pt}}
\end{center}
\end{table}
\begin{table}[!ht]
\begin{center}
\caption {({\bf Example 4.3} for subdiffusion equation) Convergence rate  with $T = 5,\kappa\geq\lambda_1$ and various $L$.
\label{T_FPDEiii} } \vspace*{0.1pt}
\def\temptablewidth{1.0\textwidth}
{\rule{\temptablewidth}{1pt}}
\begin{tabular*}{\temptablewidth}{@{\extracolsep{\fill}}cccccc}
$N$ &\multicolumn{2}{c}{$\kappa=1$}
&\multicolumn{2}{c}{$\kappa = 1.5$} \\
\cline{2-3}    \cline{4-5}
& $L = \pi$&$L = 4$ & $L=\pi$&$L=4$\\
&$(\lambda_1 = 1)$ & $(\lambda_1\approx 0.62)$&$(\lambda_1=1)$& $(\lambda_1\approx 0.62)$\\
\hline
     64&1.00&0.94&0.92&0.92\\
  128&1.00&0.96 &0.94 &0.92\\
     256 &1.00&0.97 &0.96 &0.93\\
     512 &1.00&0.98 &0.97 &0.95\\
\end{tabular*}
{\rule{\temptablewidth}{1pt}}
\end{center}
\end{table}
\section{Detailed proofs for the main results}

\subsection{The proof of Theorem  \ref{TH_BDF2ode}}
Since BDF2 is a two-step method,  the { analysis of decaying errors} is more {chanllenging  than for} the Euler and C-N schemes. {Here, we} introduce the concept of {a} discrete orthogonal convolution (DOC) kernel, {which differs} from the BDF2 convolution kernels {provided} in \cite{LZ20,ZZ21}.
\begin{definition}
For  convolution sequence $\{A^{(n)}_{n-k}\}_{k=1}^n$,  the DOC kernel in \cite{LZ20}  is defined by
     \begin{equation}
       \sum_{j=k}^{n}\theta_{n-j}^{(n)}A_{j-k}^{(j)} \equiv \delta_{nk},\quad \forall 1 \leq k \leq n,\label{THB}
     \end{equation}
 where $\delta_{nk}$ denotes the Kronecker delta symbol with $\delta_{nk} = 1$ if $n = k$ and $\delta_{nk} = 0$ if $n \neq k$.
\end{definition}
 \begin{lemma}\label{Lemma_theta}
Let $0<-\kappa\tau< 1/2$ and  the BDF2 kernels $A_{n-k}^{(n)}$ defined  by \eqref{EQ_BDF2kernels}. Then the DOC kernels  defined by \eqref{THB} can be explicitly expressed as
\begin{align}
\theta^{(n)}_{n-1} &= \frac{2-2\kappa\tau}{3-2\kappa\tau}\frac1{\sqrt{1+2\kappa\tau}}\big((2-\sqrt{1+2\kappa \tau})^{-n}-(2+\sqrt{1+2\kappa \tau})^{-n}\big),\\
\theta^{(n)}_{n-k} &=\frac1{\sqrt{1+2\kappa\tau}}\big((2-\sqrt{1+2\kappa \tau})^{-(n-k+1)}-(2+\sqrt{1+2\kappa \tau})^{-(n-k+1)}\big).
\end{align}
 \end{lemma}
 \begin{proof}
 Form the definition of the DOC kernels \eqref{THB}, one has
 \begin{align*}
 \theta^{(n)}_{0}A^{(n)}_0 &= 1,\\
  \theta^{(n)}_{0}A^{(n)}_1+ \theta^{(n)}_{1}A^{(n-1)}_0 &= 0,\\
   \theta^{(n)}_{n-j}A^{(j)}_0+ \theta^{(n)}_{n-j-1}A^{(j+1)}_1+ \theta^{(n)}_{n-j-2}A^{(j+2)}_2 &= 0,\quad j\leq n-2.
 \end{align*}
Set $a_j:=\theta^{(n)}_{n-j}A^{(j)}_0 \;(1\leq j\leq n)$ and $\phi :=(3/2-\kappa\tau)^{-1}$. Then $a_j$ solves
  \begin{equation*}
a_n = 1,\quad  a_{n-1}= 2\phi, \quad
  a_j= 2\phi a_{j+1}-\frac{1}{2} \phi a_{j+2} \;( j\leq n-2),
 \end{equation*}
 where the  definition \eqref{EQ_BDF2kernels} is used.
It is easy to check $\phi\in(\frac12,\frac23)$ by the fact $0<-\kappa\tau<\frac12$. Let $c = -\phi-\sqrt{\phi(\phi-\frac12)}$ be the solution to equation $x^2+2\phi x+\phi/2=0$ and $b = c+2\phi$. Then for $j\leq n-2$, one finds
 \begin{equation}
 a_j+c a_{j+1} = b(a_{j+1}+ca_{j+2})=\cdots=b^{n-j-1}(a_{n-1}+c a_n) = b^{n-j}.\label{EQ_1712}
 \end{equation}
 Let $d = -\frac{b}{b+c}$. It is easy to verify that \eqref{EQ_1712} can be reformulated into
 \begin{equation*}
 \frac{a_j}{b^{n-j}}+d = (-\frac{c}{b})( \frac{a_{j+1}}{b^{n-j-1}}+d )=\cdots= (-\frac{c}{b})^{n-1-j}( \frac{a_{n-1}}{b}+d ).
 \end{equation*}
 Thus, we finally arrive at
 \begin{equation*}
   a_j = -\frac1{b+c}\big((-c)^{n-j+1}-b^{n-j+1}\big),
 \end{equation*}
where $ b=2\phi+c = \frac{1}{2+\sqrt{1+2\kappa\tau}}$ and
 \begin{equation*}
 -c = \phi+\sqrt{\phi(\phi-\frac12)} = \frac{\phi/2}{\phi-\sqrt{\phi(\phi-\frac12)}} = \frac{1}{2-\sqrt{4-2/\phi}}=\frac{1}{2-\sqrt{1+2\kappa\tau}}.
 \end{equation*}
 The proof is completed by the identity $\theta^{(n)}_{n-j}:=a_j /A^{(j)}_0(1\leq j\leq n)$.
 \end{proof}

With the help of the new DOC kernels,  Theorem \ref{TH_BDF2ode} can be proven as follows.

{\bf Proof of Theorem \ref{TH_BDF2ode}}\\
\begin{proof}
Multiplying Eq. \eqref{EQ_CF} by the DOC kernels and summing over $k$, one has
\begin{align}
\sum_{k=1}^n \theta^{(n)}_{n-k}\sum_{j = 1}^k A^{(n)}_{n-k} e^j = \sum_{k=1}^n\theta^{(n)}_{n-k}\tau R^{k}.
\end{align}
Applying  the definition of DOC kernels \eqref{THB} to the left hand, one finds
\begin{align}
  \sum_{k=1}^n \theta^{(n)}_{n-k}\sum_{j = 1}^k A^{(n)}_{n-j} e^j  = \sum_{j=1}^n e^j \sum_{k = j}^n \theta^{(n)}_{n-k}A^{(n)}_{n-j} = e^n.
\end{align}
Hence, we arrive at
\begin{align}
 e^n = \sum_{k=1}^n\theta^{(n)}_{n-k}\tau R^{k}= \sum_{k=1}^n\theta^{(n)}_{n-k}\tau \eta^{k}+(\theta^{(n)}_{n-2}/2-\theta^{(n)}_{n-1})e^0,\label{EQ_1552}
\end{align}
where \eqref{EQ_RRR} is used.
 It is easy to check
 $$(2-\sqrt{1+2\kappa \tau})^{-(n-k+1)}=(1-\frac{2\kappa \tau}{1+\sqrt{1+2\kappa \tau}})^{-(n-k+1)}\leq (1-\kappa \tau)^{-(n-k+1)}.$$
Then for $\tau<1/(4\lambda)$,  combining with Lemma \ref{Lemma_theta}, one produces
 \begin{align}
 0<\theta^{(n)}_{n-k} \leq 2 (1-\kappa \tau)^{-(n-k+1)}. \label{EQ_2350}
 \end{align}
 Inserting the inequality \eqref{EQ_2350} into \eqref{EQ_1552}, one yields
 \begin{align}
 |e^n| \leq & (1-\kappa\tau)^{-n}|e^0|+2\sum_{k=1}^n  \tau(1-\kappa \tau)^{-(n-k+1)} |\eta ^{k}|\nonumber\\
 \leq & (1-\kappa\tau)^{-n}|e^0|+2C_{u,\alpha}\big(\sum_{k=2}^{n-2}  \tau^3(1-\kappa \tau)^{-(n-k-1)} t_k^{\alpha-3}+4(1-\kappa\tau)^{-n}\tau^\alpha\big),  \nonumber
\end{align}
 where the last inequality uses estimates in \eqref{EQ_0018} and the fact $\tau \leq 1/(-4\kappa )$.

Set $\rho_\tau =(1-\kappa\tau)$ and $S_n = \sum_{k=2}^{n-2}  \tau\rho_\tau^{-(n-k-1)} t_k^{\alpha-3}$. Then, we has
\begin{align}
|e^n|\leq \rho_\tau^{-n}|e^{0}|+2 C_{u,\alpha} \Big(S_n\tau^2 + 4\rho_\tau^{-n}\tau^\alpha\Big).\label{EQ_0024}
\end{align}
It follows from Lemma \ref{Lemma_SCN} that
\begin{align}
S_n&\leq \frac{1}{2-\alpha}\Big((2^{2-\alpha}(1-\rho_\tau^{-(\frac{n}{2}-1)})-1)t_{n-1}^{\alpha-2}+\rho_\tau^{-(\frac{n}{2}-2)}\tau^{\alpha-2}\Big).\label{EQ_0046}
\end{align}
Applying Lemma \ref{Pro_1} to \eqref{EQ_0046} and noting $\rho_\tau^2\leq 2$, one has
\begin{align}
S_n&\leq \frac{1}{2-\alpha}\Big((2^{2-\alpha}(1-e^{\frac{\kappa  t_n}{2}})-1)t_{n-1}^{\alpha-2}+2e^{\frac{\kappa  t_n}{4}}\tau^{\alpha-2}\Big).\label{EQ_2046}
\end{align}
Thus, the proof is completed by inserting \eqref{EQ_2046} into \eqref{EQ_0024} and using  Lemma \ref{Pro_1}.
\end{proof}

\subsection{The proofs of the main results in Section \ref{Sec_3} }
The following lemma will be widely used in the  proofs.
\begin{lemma}\label{Lemma_abc} We here present the following discrete and continuous estimates:
\begin{enumerate}
  \item Let $\bm{a}=(a_k)_{k=1}^\infty,\bm{b}^j=(b_k^j)_{k=1}^\infty\in l^2,\;1\leq j\leq m$ for some positive integers $m$. If $a_k\leq \sum_{j=1}^m b_k^j$ for all $1\leq k<\infty$, then it holds that
\begin{align}
\sqrt{\sum_{k=1}^\infty a_k^2}\leq \sum_{j=1}^m\sqrt{\sum_{k=1}^\infty (b_k^j)^2}.\label{EQ_0017}
\end{align}
 \item Let $v_k(x)\in L^2((a,b)), 1\leq k<\infty$ for some fixed interval $(a,b)$. Assume there exists $v^*(x)\in L^2((a,b))$ such that the series $\sum_{k=1}^\infty v_k^2(x)\leq v^*(x), a. e.$. Then, it holds that
     \begin{align}
\sum_{k=1}^\infty(\int_a^b v_k(s)\d s)^2 \leq (\int_a^b \sqrt{\sum_{k=1}^\infty v_k^2(s)}\d s)^2.\label{EQ_2039}
\end{align}
\end{enumerate}
\end{lemma}
\begin{proof}
Note that the $l^2$-norm is given by
$\|\bm{\upsilon}\|_{l^2} = \sqrt{\sum_{k=1}^\infty (\upsilon_k)^2}$ for all $\bm{\upsilon}=(\upsilon_k)_{k=1}^\infty\in l^2$. Thus, the first claim can be immediately derived by the triangle inequality of $l^2$-norm and the assumption $a_k\leq \sum_{j=1}^m b_k^j$.
Note that
\begin{align*}
\sum_{k=1}^n(\int_a^b v_k(s)\d s)^2 &= \int_a^b \int_a^b(\sum_{k=1}^nv_k(s)v_k(r))\d s\d r\\
&\leq \int_a^b \int_a^b\sqrt{\sum_{k=1}^nv_k^2(s)}\sqrt{\sum_{k=1}^nv_k^2(r)}\d s\d r\\
&=(\int_a^b \sqrt{\sum_{k=1}^nv_k^2(s)}\d s)^2.
\end{align*}
The proof of the second claim is completed by taking $n\to \infty$ and using the dominated convergence theorem.
\end{proof}

{\bf The proof of Theorem  \ref{TH_IEpde}}\\
\begin{proof}
{From \eqref{EQ_basis},}  the solutions to \eqref{EQ_exact} and \eqref{EQ_pdeIE} has the following representations
\begin{align}
u = \sum_{k=1}^\infty d_k(t)\omega_k(x), \label{EQ_du}\\
U^n = \sum_{k=1}^\infty d_k^n\omega_k(x),\label{EQ_dU}
\end{align}
where $d_k(t) = (u,\omega_k)$ and $d_k^n = (U^n,\omega_k)$.
Taking inner products of \eqref{EQ_pdeIEerror} with $\omega_k$ yields
\begin{align}
\nabla_\tau e_k^n &= -(\lambda_k-\kappa)\tau e_k^n+\tau R^n,\label{EQ_0043}
\end{align}
where $e_k^n:=d_k(t_n)-d_k^n$, $R^n_k = -\frac1{\tau_n}\int_{t_{n-1}}^{t_n}(s-t_{n-1})\partial_{tt}d_k(s)\d s$ and {$\lambda_k$ is the eigenvalue of eigenvalue problem \eqref{EQ_eigen}}.
Applying Lemma \ref{Lemma_abc} to \eqref{EQ_0043}, we arrive at
\begin{align}
\sqrt{\sum_{k=1}^\infty(1+(\lambda_k-\kappa)\tau)^2|e_k^n|^2}\leq \sqrt{\sum_{k=1}^\infty|e_k^{n-1}|^2}+\tau\sqrt{\sum_{k=1}^\infty|R_k^n|^2}.\label{EQ_1511}
\end{align}
Note that $\|u\|^2 = \sum_{k=1}^\infty |d_k|^2$ and $\|e^n\|^2 = \sum_{k=1}^\infty |e^n_k|^2$.
 Applying Lemma \ref{Lemma_abc} again to the last term of \eqref{EQ_1511}, one has
 \begin{align}
  \sqrt{\sum_{k=1}^\infty|R_k^n|^2}\leq \frac1{\tau_n}\int_{t_{n-1}}^{t_k}(s-t_{n-1})\|u_{tt}\|\d s:=\hat{R}^n.
 \end{align}
Hence, we arrive at
\begin{align}
(1+(\lambda_1-\kappa)\tau)\|e^n\|\leq \|e^{n-1}\|+\tau \hat{R}^n,
\end{align}
which implies
\begin{align}
\|e^n\|\leq (1+(\lambda_1-\kappa)\tau)^{-n}\|e^{0}\|+\tau \sum_{k=1}^n (1+(\lambda_1-\kappa)\tau)^{-(n+1-k)}\hat{R}^k.
\end{align}
The {remaining}  proof is similar to {that of }Theorem \ref{TH_IEode}, and we omit {it} here.
\end{proof}

{\bf The proof of Theorem  \ref{TH_CNpde}} \\
 \begin{proof}
Taking inner products of \eqref{EQ_0100} with $\omega_k$, one has
\begin{align}
(1+\frac{\lambda_k-\kappa}2\tau) e_k^n =(1-\frac{\lambda _k-\kappa}2\tau) e_k^{n-1}+R_k^n,
\end{align}
where $R_k^n:=(R^n,\omega)$
and $R^n = R_1^n+R_2^n$. It follows from \eqref{EQ_du}, \eqref{EQ_dU} that
$$\|e^n\| =\sum_{k=1}^\infty |e_k^n|^2.$$
Note that $\|u\|^2=\sum_{k=1}^\infty d_k(t)^2$, then from \eqref{EQ_2231}, \eqref{EQ_2232}, \eqref{EQ_0017} and   \eqref{EQ_2039}, one finds
\begin{align}
 \sqrt{\sum_{n=1}^\infty |R_k^n|^2}\leq \hat{R}^n.\label{EQ_0105}
\end{align}
For simplicity of notation, denote by   $\psi_k :=\frac{1+\frac{(\lambda_k-\kappa)\tau}{2}}{1-\frac{(\lambda_k-\kappa)\tau}{2}}$. Then, we arrive at
\begin{align}
|e_k^n|\leq \psi_k^{-n}|e_k^{0}|+\tau  (1+\frac{(\lambda_k-\kappa)\tau}{2})^{-1} \sum_{j=1}^n \psi_j^{-(n-j)}R_k^j.
\end{align}
Squaring both sides and summing $k$ from 1 to $\infty$, it follows from \eqref{EQ_0017} that
\begin{align}
\|e^n\|\leq \psi_1^{-n}\|e^{0}\|+\tau  (1+\frac{(\lambda_1-\kappa)\tau}{2})^{-1} \sum_{j=1}^n \psi_1^{-(n-j)}\hat{R}^j,
\end{align}
where $\hat{R}^j$ satisfies \eqref{EQ_2212}.
The rest proof is similar to the proof of Theorem \ref{TH_CNode} and we omit here.
\end{proof}

{\bf The proof of Theorem  \ref{TH_BDF2pde}}\\
\begin{proof}
Taking inner products of both sides with $\omega_k$, one has
\begin{align}
e_k^1-e_k^{0}&=-(\lambda_k-\kappa)\tau  e_k^{1}+\tau \eta_k^{1},&&x\in \Omega,\;\label{EQ_2044}\\
\frac32e_k^n-2e_k^{n-1}+\frac12e_k^{n-2}&=-(\lambda_k-\kappa)\tau e_k^{n}+\tau \eta_k^{n}, && x\in \Omega,\;2\leq n\leq N,\label{EQ_2045}
\end{align}
where $e_k^n:=d_k(t_n)-d_k^n$, $\eta^n_k =(\eta^n,\omega_k)$ and {$\lambda_k$ is the eigenvalue of eigenvalue problem \eqref{EQ_eigen}}. It follows from \eqref{EQ_du}, \eqref{EQ_dU}  that
\begin{align}
\|e^n\| =& \sum_{k=1}^\infty |e_k^n|^2.\label{EQ_0014}
\end{align}
Denote by $\hat{\eta}^n:= \sum_{n=1}^\infty |\eta_k^n|^2$. Similar to \eqref{EQ_0105}, from \eqref{EQ_0017}, \eqref{EQ_2039} and assumption \eqref{EQ_Fregularity}, one has
\begin{align}
\hat{\eta}^n&\leq C_{u,\alpha} t_{k-2}^{\alpha-3}\tau^2,\quad\; n\geq 3\label{EQ_0015}\\
 \hat{\eta}^n &\leq C_{u,\alpha}\tau^{\alpha-1},\qquad   n=1,2.\label{EQ_0016}
\end{align}
Set $R_k^l = \eta_k^l(l\geq 3)$ and $R_k^2 = \eta_k^2-\frac{e_k^0}{2\tau} ,\;R_k^1 = \eta_k^1+\frac{e_k^0}{\tau}$. We further introduce the following notations, say the BDF2 kernels, as $A^{(k,1)}_0 :=1+(\lambda_k-\kappa)\tau$ and for $l\geq 2$
$$A^{(k,l)}_0 :=\frac32+(\lambda_k-\kappa)\tau, \;A^{(k,l)}_1 := -2,\;A^{(k,l)}_2 := \frac12,\;A^{(k,l)}_j = 0,\quad j\geq 3.$$
The error identities  \eqref{EQ_2044} and \eqref{EQ_2045} may be refined as a convolution form
\begin{align}
\sum_{j=1}^lA^{(k,l)}_{l-j} e_k^j = R_k^l.\label{EQ_2053}
\end{align}

Let $\theta^{(k,n)}_{n-j}$ be the DOC kernels corresponding to the BDF2 kernels $A^{(k,n)}_{n-j}$. Multiplying  both sides of  \eqref{EQ_2053} by the DOC kernels $\theta^{(k,n)}_{n-j}$ and summing over $j$, one has
\begin{align}
 e_k^n = \sum_{j=1}^n\theta^{(k,n)}_{n-j}\tau R_k^{j}= \sum_{j=1}^n\theta^{(k,n)}_{n-j}\tau \eta_k^{j}+(\theta^{(k,n)}_{n-2}/2-\theta^{(k,n)}_{n-1})e_k^0 .\label{EQ_2109}
\end{align}
Similar to  \eqref{EQ_2350}, for $\tau<1/(4(\lambda_k-\kappa))$, the DOC kernels satisfy
 \begin{align}
 0<\theta^{(k,n)}_{n-j} \leq 2 (1+(\lambda_k-\kappa) \tau)^{-(n-j+1)}\leq 2 (1+(\lambda_1-\kappa) \tau)^{-(n-j+1)}. \label{EQ_2107}
 \end{align}
 Inserting \eqref{EQ_2107} into \eqref{EQ_2109}, one arrive at
 \begin{align}
 |e_k^n| \leq  (1+(\lambda_1-\kappa)\tau)^{-n}|e_k^0|+2\sum_{j=1}^n  \tau(1+(\lambda_1-\kappa) \tau)^{-(n-j+1)} |\eta_k^{j}|.
\end{align}
Squaring both sides  and summing from 1 to $\infty$,
combining with   \eqref{EQ_0017}, \eqref{EQ_0014}-\eqref{EQ_0016}, one has
\begin{align}
 \|e^n\| \leq  (1+(\lambda_1-\kappa)\tau)^{-n}\|e^0\|+2\sum_{j=1}^n  \tau(1+(\lambda_1-\kappa)\tau)^{-(n-j+1)} \hat{\eta}^{j}.
\end{align}
The rest proof is similar to the proof of Theorem \ref{TH_BDF2pde} and we omit here.
\end{proof}

\section{Conclusion}
{When employing} popular numerical schemes to solve diffusion and subdiffusion equations with weakly initial singularity \eqref{EQ_Fregularity}, {various} convergence orders can be  observed.
In this paper,
we {introduce} a general methodology  and  {propose a} new decay-preserving error estimate to systematically {analyze the}  numerical behavior of  widely used IE, C-N and BDF2 schemes.  The decay-preserving error estimate {comprises}   two { components: $e^{-C(\lambda_1-\kappa) T}\tau^{\alpha}$} and $T^{\alpha-k}\tau^k$, where $\lambda_1$ is the minimal eigenvalue of {the} problem \eqref{EQ_eigen}.  {This estimate reveals that the diverse} convergence rates {arise due to} the  trade-off  between {these} two components in different model parameter regimes.  Our decay-preserving error estimates {successfully capture varying convergence states, overcoming limitations  of} traditional error estimates {by incorporating}  model parameters, retaining more properties of continuous equations.  Furthermore, we present a conjecture {regarding}  the decay-preserving error estimate of {the}  L1 scheme for sub-diffusion equations \eqref{EQ_FPDEs}. This conjecture {sheds light on} the phenomena {where} the  L1 scheme may {exhibit} different convergence rates in {distinct} model parameter regimes. {We provide numerical results}  to validate our analysis.

{Nevertheless}, our  decay-preserving error estimates {may not account for certain numerical phenomena, such as the observed kinks} in  the convergence rate {transition from} high-order to low-order {as depicted}  in  Figure \ref{Fig_PDEtick}. {This suggests the possibility of a}  sup-convergence point $N_0$ {where}  the $L^2$-norm errors {experience a dramatic drop. Consequently}, a more {sophisticated} estimate is {needed} to interpret these {intriguing} phenomena in the future.

\section*{Acknowledgements}
{  This work is supported in part by the National Natural Science Foundation of China under grants Nos. 12171376, 11871092, 12131005, 2020-JCJQ-ZD-029, the }Natural Science Foundation of Hubei Province No. 2019CFA007, and the Fundamental Research Funds for the Central Universities 2042021kf0050. The numerical simulations in this work have been done on the supercomputing system in the Supercomputing Center of Wuhan University.

\bibliographystyle{plain}
\bibliography{IS}

\begin{thebibliography}{10}

\bibitem{D10}
K.~Diethelm.
\newblock The analysis of fractional differential equations: An
  application-oriented exposition using differential operators of caputo type.
\newblock {\em Springer Berlin, Heidelberg}, 2010.

\bibitem{GOS18}
J.~Gracia, E.~O'Riordan, and M.~Stynes.
\newblock Convergence in positive time for a finite difference method applied
  to a fractional convection-diffusion problem.
\newblock {\em Comput. Meth. Appl. Mat.}, 18(1):33--42, 2018.

\bibitem{JLZ16}
B.~Jin, R.~Lazarov, and Z.~Zhou.
\newblock An analysis of the {L1} scheme for the subdiffusion equation with
  nonsmooth data.
\newblock {\em IMA J. Numer. Anal.}, 36(1):197--221, 2016.

\bibitem{JLZ17}
B.~Jin, B.~Li, and Z.~Zhou.
\newblock Correction of high-order bdf convolution quadrature for fractional
  evolution equations.
\newblock {\em SIAM J. Sci. Comput.}, 39(6):A3129--A3152, 2017.

\bibitem{JLZ18}
B.~Jin, B.~Li, and Z.~Zhou.
\newblock Numerical analysis of nonlinear subdiffusion equations.
\newblock {\em SIAM J. Numer. Anal.}, 56(1):1--23, 2018.

\bibitem{JK2011}
J.~Klafter and I.~Sokolov.
\newblock First steps in random walks: from tools to applications.
\newblock {\em OUP Oxford}, 2011.

\bibitem{K21}
N.~Kopteva.
\newblock Error analysis of an {L2}-type method on graded meshes for a
  fractional-order parabolic problem.
\newblock {\em Math. Comp.}, 90(327):19--40, 2021.

\bibitem{LQZ21}
D.~Li, H.~Qin, and J.~Zhang.
\newblock Sharp pointwise-in-time error estimate of {L1} scheme for nonlinear
  subdiffusion equations.
\newblock 2021. arXiv preprint arXiv:2101.04554.

\bibitem{LLZ18}
H.~Liao, D.~Li, and J.~Zhang.
\newblock Sharp error estimate of the nonuniform {L1} formula for linear
  reaction-subdiffusion equations.
\newblock {\em SIAM J. Numer. Anal.}, 56(2):1112--1133, 2018.

\bibitem{LMZ19}
H.~Liao, W.~McLean, and J.~Zhang.
\newblock A discrete gronwall inequality with applications to numerical schemes
  for subdiffusion problems.
\newblock {\em SIAM J. Numer. Anal.}, 57(1):218--237, 2019.

\bibitem{LZ20}
H.~Liao and Z.~Zhang.
\newblock Analysis of adaptive {BDF2} scheme for diffusion equations.
\newblock {\em Math. Comp.}, 90(329):1207--1226, 2020.

\bibitem{LX07}
Y.~Lin and C.~Xu.
\newblock Finite difference/spectral approximations for the time-fractional
  diffusion equation.
\newblock {\em J. Comput. Phys.}, 225(2):1533--1552, 2007.

\bibitem{L88}
C.~Lubich.
\newblock Convolution quadrature and discretized operational calculus. {I}.
\newblock {\em Numer. Math.}, 52(2):129--145, 1988.

\bibitem{MM15}
W.~McLean and K.~Mustapha.
\newblock Time-stepping error bounds for fractional diffusion problems with
  non-smooth initial data.
\newblock {\em J. Comput. Phys.}, 293:201--217, 2015.

\bibitem{Metzler2014Anomalous}
R.~Metzler, J.~Jeon, A.~Cherstvy, and E.~Barkai.
\newblock Anomalous diffusion models and their properties: non-stationarity,
  non-ergodicity, and ageing at the centenary of single particle tracking.
\newblock {\em Phys. Chem. Chem. Phys.}, 16(44):24128--24164, 2014.

\bibitem{Ralf2000The}
R.~Metzler and J.~Klafter.
\newblock The random walk's guide to anomalous diffusion: a fractional dynamics
  approach.
\newblock {\em Phys. Rep.}, 339(1):1--77, 2000.

\bibitem{SOG17}
M.~Stynes, E.~O'Riordan, and J.L. Gracia.
\newblock Error analysis of a finite difference method on graded meshes for a
  time-fractional diffusion equation.
\newblock {\em SIAM J. Numer. Anal.}, 55:1057--1079, 2017.

\bibitem{SW06}
Z.~Sun and X.~Wu.
\newblock A fully discrete scheme for a diffusion wave system.
\newblock {\em Appl. Numer. Math.}, 56(2):193--209, 2006.

\bibitem{G06}
V.~Thom\'{e}e.
\newblock {\em Galerkin finite element methods for parabolic problems, second
  edition}.
\newblock Springer Berlin, Heidelberg, 2006.

\bibitem{YKF18}
Y.~Yan, M.~Khan, and N.~Ford.
\newblock An analysis of the modified {L1} scheme for time-fractional partial
  differential equations with nonsmooth data.
\newblock {\em SIAM J. Numer. Anal.}, 56(1):210--227, 2018.

\bibitem{ZZ21}
J.~Zhang and C.~Zhao.
\newblock Sharp error estimate of {BDF2} scheme with variable time steps for
  linear reaction-diffusion equations.
\newblock {\em J. Math.}, 41(6):471--488, 2021.

\end{thebibliography}
%
%
%
%
%

\end{document}